\documentclass[12pt]{amsart}
\usepackage{amscd,amsfonts,amssymb}
\usepackage{tikz}
\usepackage{pgfplots}
\usetikzlibrary{automata,positioning,calc,trees}
\usetikzlibrary{intersections,pgfplots.fillbetween}
\usepackage{graphicx,tikz}
\usepackage{pgf,tikz,pgfplots}
\usepackage{mathrsfs}
\usetikzlibrary{arrows}
\usepackage{enumerate}
\usepackage[shortlabels]{enumitem}
\usepackage{mathrsfs}
\usepackage{amssymb,amsmath,amsthm,color}
\usepackage{caption,subcaption}
\usepackage[alphabetic,bibtex-style]{amsrefs}
\usepackage{hyperref}
\usepackage{url}
\usepackage{setspace}
\usepackage{float}

\BibSpec{article}{%
	+{}  {\PrintAuthors}                {author}
	+{,} { \textit}                     {title}
	+{.} { }                            {part}
	+{:} { \textit}                     {subtitle}
	+{,} { \PrintContributions}         {contribution}
	+{.} { \PrintPartials}              {partial}
	+{,} { }                            {journal}
	+{}  { \textbf}                     {volume}
	+{}  { \PrintDatePV}                {date}
	+{,} { \issuetext}                  {number}
	+{,} { \eprintpages}                {pages}
	+{,} { }                            {status}
	+{,} { \url}                        {url}    % <---- ADDED
	+{,} { \PrintDOI}                   {doi}
	+{,} { available at \eprint}        {eprint}
	+{}  { \parenthesize}               {language}
	+{}  { \PrintTranslation}           {translation}
	+{;} { \PrintReprint}               {reprint}
	+{.} { }                            {note}
	+{.} {}                             {transition}
}

\date{\today}
\allowdisplaybreaks
\newcommand{\R}{{\mathbb {R}}}

\newcommand{\N}{{\mathbb N}}
\newcommand{\Z}{{\mathbb Z}}

\newcommand{\s}{\mathbb S}

\newcommand{\vertiii}[1]{{\left\vert\kern-0.25ex\left\vert\kern-0.25ex\left\vert #1 
		\right\vert\kern-0.25ex\right\vert\kern-0.25ex\right\vert}}

\newtheorem{theorem}{Theorem}[section]
\newtheorem{lemma}[theorem]{Lemma}
\newtheorem{remark}[theorem]{Remark}

\newtheorem{coro}[theorem]{Corollary}
\newtheorem{proposition}[theorem]{Proposition}

\newtheorem*{theorem*}{Theorem}

\theoremstyle{definition}

\numberwithin{equation}{section}

\begin{document}
\title{Sharp endpoint $L^p-$estimates for bilinear spherical maximal functions}

\author[Bhojak]{Ankit Bhojak}
\address{Ankit Bhojak\\
	Department of Mathematics\\
	Indian Institute of Science Education and Research Bhopal\\
	Bhopal-462066, India.}
\email{ankitb@iiserb.ac.in}

\author[Choudhary]{Surjeet Singh Choudhary}
\address{Surjeet Singh Choudhary\\
	Department of Mathematics\\
	Indian Institute of Science Education and Research Bhopal\\
	Bhopal-462066, India.}
\email{surjeet19@iisermohali.ac.in, surjeet19@iiserb.ac.in}

\author[Shrivastava]{Saurabh Shrivastava}
\address{Saurabh Shrivastava\\
	Department of Mathematics\\
	Indian Institute of Science Education and Research Bhopal\\
	Bhopal-462066, India.}
\email{saurabhk@iiserb.ac.in}

\author[Shuin]{Kalachand Shuin}
\address{Kalachand Shuin\\
	Department of Mathematics\\
	Indian Institute of Science\\
	Bengaluru-560012, India.}
\email{shuin.k.c@gmail.com, kalachands@iisc.ac.in}

\thanks{}
\begin{abstract}
{In this article, we address endpoint issues for the bilinear spherical maximal function. We obtain borderline restricted weak type estimates for the well studied bilinear spherical maximal function
$$\mathfrak{M}(f,g)(x):=\sup_{t>0}\left|\int_{\mathbb S^{2d-1}}f(x-ty_1)g(x-ty_2)\;d\sigma(y_1,y_2)\right|,$$
in dimensions $d=1,2$ and as an application, we deduce sharp endpoint estimates for the multilinear spherical maximal function. We also prove $L^p-$estimates for the local spherical maximal function in all dimensions $d\geq 2$, thus improving the boundedness left open in the work of Jeong and Lee (https://doi.org/10.1016/j.jfa.2020.108629).
We further study necessary conditions for the bilinear maximal function, \[\mathcal M (f,g)(x)=\sup_{t>0}\left|\int_{\mathbb S^{1}}f(x-ty)g(x+ty)\;d\sigma(y)\right|\]
to be bounded from $L^{p_1}(\mathbb R^2)\times L^{p_2}(\mathbb R^2)$ to $L^p(\mathbb R^2)$ and prove sharp results for a bilinearized version of $\mathcal M$.}
\end{abstract}
\maketitle
\addtocontents{toc}{\sloppy}
\section{Introduction and main results}
Let $\sigma$ be the surface measure on the $(d-1)$-dimensional unit sphere $\s^{d-1}$. The study of $L^p-$improving properties of the spherical averages defined as
\[A_tf(x)=\int_{\s^{d-1}}f(x-ty)\;d\sigma(y),\]
was initiated by Littman \cite{Littman} and Strichartz \cite{Strichartz}. Consider the spherical maximal function defined as
\[A_*f(x)=\sup_{t>0}\left|A_tf(x)\right|.\]
For $p>\frac{d}{d-1}$, the $L^p-$boundedness of $A_*$ was established by Stein \cite{MaximalFunctionsISphericalMeans} for dimension $d\geq3$ and Bourgain \cite{BourgainCircular} in dimension two. At the endpoint $p=\frac{d}{d-1}$, the restricted weak type inequality was proved by Bourgain for dimensions $d\geq3$. To study the case of dimension two, Oberlin \cite{OberlinLinearization} considered the following linearized version of the spherical maximal operator 
\[\widetilde Af(x)=\int_{\s^{d-1}}f(x-|x|y)\;d\sigma(y),\]
and showed that $\widetilde A$ maps $L^{2,1}(\R^2)$ to $L^{2,\infty}(\R^2)$ boundedly. However, using a modification of the Kakeya construction from \cite{Keich}, it was finally shown in \cite{EndpointMappingPropertiesOfSphericalMaximalOperators} that the spherical maximal operator $A_*$ does not map $L^{2,1}(\R^2)$ to $L^{2,\infty}(\R^2)$.

\sloppy The local spherical maximal operator $A_{loc}f(x)=\sup_{1\leq t\leq2}\left|A_tf(x)\right|$ has also been widely studied. Schlag and Sogge \cite{Schlag, SchlagSogge} showed that $A_{loc}$ is bounded from $L^p(\R^d)$ to $L^q(\R^d)$ when $d\geq 2$ and $\left(\frac{1}{p},\frac{1}{q}\right)$ lies in the interior of the closed convex hull generated by the points $(0,0),\;\left(\frac{d-1}{d},\frac{d-1}{d}\right),\;\left(\frac{d-1}{d},\frac{1}{d}\right)$ and $\left(\frac{d^2-d}{d^2+1},\frac{d-1}{d^2+1}\right)$ and unbounded if $\left(\frac{1}{p},\frac{1}{q}\right)$ lies outside the closed convex hull. Moreover, boundedness on the boundary of the hull was resolved by Lee \cite{Lee1} for all dimensions $d\geq2$. The sparse bounds for the operator $A_{*}$ were obtained by Lacey \cite{SparseBoundsForSphericalMaximalFunction}. We also refer to \cite{SeegerVariation} for sparse domination of the corresponding $r-$variation operators related to the family of spherical averages.

\subsection{Part I}
\subsubsection{Full bilinear spherical maximal function} In this part we deal with the well-studied bilinear variant of the spherical maximal function defined as
$$\mathfrak{M}(f,g)(x):=\sup_{t>0}\left|\int_{\mathbb S^{2d-1}}f(x-ty_1)g(x-ty_2)\;d\sigma(y_1,y_2)\right|.$$
This operator first appeared in the work \cite{GGIPS}. The optimal strong type bounds were obtained in \cite{MaximalEstimatesForTheBilinearSphericalAveragesAndTheBilinearBochnerRieszOperators} by a method of ``slicing". The authors showed that for $d\geq 2$, the operator $\mathfrak{M}$ maps $L^{p_1}(\R^d)\times L^{p_2}(\R^d)$ to $L^{p}(\R^d)$ if and only if $p>\frac{d}{2d-1}$ and $\frac{1}{p_1}+\frac{1}{p_2}=\frac{1}{p}$ except the points on the lines $p_1=1$ and $p_2=1$. Moreover, the restricted weak type estimate $\mathfrak{M}:L^{p_1,1}(\R^d)\times L^{p_2,1}(\R^d)\to L^{\frac{d}{2d-1},\infty}(\R^d)$ holds for dimensions $d\geq3$; however, the endpoint boundedness in dimension two remained open. 

Now, we discuss the case of dimension one. The boundedness of $\mathfrak{M}$, when $d=1$, was studied in \cite{ChristZhou, DosidisRamos}. They proved that $\mathfrak{M}:L^{p_1}(\R)\times L^{p_2}(\R)\to L^p(\R)$ for $p_1,p_2>2$ and $\frac{1}{p_1}+\frac{1}{p_2}=\frac{1}{p}$. Moreover, the weak type estimate fails at the lines $p_1=2$ and $p_2=2$. 

The boundedness of the lacunary analogue of $\mathfrak{M}$ was studied in \cite{Borges} in dimensions $d\geq2$ and \cite{ChristZhou} in dimension one. Also see \cite{LeeShuin} for boundedness of bilinear maximal functions defined with degenerate surfaces.

Our first main result addresses the restricted weak type bounds for $\mathfrak{M}$ at the respective endpoints in dimensions one and two. We have the following,

\begin{figure}[H]
	\centering
		\begin{tikzpicture}[scale=2.8]
		\fill[gray!40] (0,0)--(0.5,0)--(0.5,0.5)--(0,0.5)--cycle;
		\draw[thin][->]  (0,0)node[left]{$O$} --(1.15,0) node[right]{$\frac{1}{p_1}$};
		\draw[thin][->]  (0,0) --(0,1.15) node[above]{$\frac{1}{p_2}$};
		\draw[dotted] (0,1)--(1,1);
		\draw[dotted] (1,0)--(1,1);
		\draw[black] (0.5,0)--(0.5,0.5)node[right]{$U_2(\frac{1}{2},\frac{1}{2})$};
		\draw[black] (0.5,0.5)--(0,0.5);
		\filldraw [black] (0.5,0)node[below]{$U_1(\frac{1}{2},0)$};
		\filldraw [black] (0,0.5)node[left]{$U_3(0,\frac{1}{2})$};
		\end{tikzpicture}
		\begin{tikzpicture}[scale=2.8]
		\fill[gray!40] (0,0)--(1,0)--(1,0.5)--(0.5,1)--(0,1)--cycle;
		\draw[thin][->]  (0,0)node[left]{$O$} --(1.15,0) node[right]{$\frac{1}{p_1}$};
		\draw[thin][->]  (0,0) --(0,1.15) node[above]{$\frac{1}{p_2}$};
		\draw[dotted] (0,1)--(1,1);
		\draw[dotted] (1,0)--(1,1);
		\draw[black] (1,0.5)--(0.5,1);
		\filldraw [black] (1,0)node[below]{(1,0)};
		\filldraw [black] (1,0.5)node[right]{$V_1(1,\frac{1}{2})$};
		\filldraw [black] (0.5,1)node[above]{$V_2(\frac{1}{2},1)$};
		\node at (1,0.5) {$\circ$};
		\node at (0.5,1) {$\circ$};
		\filldraw [black] (0,1)node[left]{(0,1)};
		\end{tikzpicture}
	\caption{The figure denotes the region of $L^{p_1,1}\times L^{p_2,1}\to L^{p,\infty}$ boundedness of $\mathfrak{M}$ on the line segments $U_1U_2\cup U_2U_3$ and $V_1V_2$ in dimensions one and two respectively.}
	\label{Fig:slicedbilineardiagonal}
\end{figure}
\begin{theorem}\label{fullmaximal}
	Let $1\leq p_1,p_2\leq\infty$. Then, the following are true.
	\begin{enumerate}[(A)]
		\item For $1\leq p\leq\infty$ with $\frac{1}{p_1}+\frac{1}{p_2}=\frac{1}{p}$, and either $p_1=2$ or $p_2=2$, we have
		\[\mathfrak{M}:L^{p_1,1}(\R)\times L^{p_2,1}(\R)\to L^{p,\infty}(\R).\]\label{bilinearone}
		\item For $1<p_1,p_2<2$ with $\frac{1}{p_1}+\frac{1}{p_2}=\frac{3}{2}$, we have
		\begin{equation*}
			\mathfrak{M}:L^{p_1,1}(\R^2)\times L^{p_2,1}(\R^2)\to L^{\frac{2}{3},\infty}(\R^2).	
		\end{equation*}\label{bilineartwo}
	\end{enumerate}
\end{theorem}
We now discuss some applications to multilinear spherical averages. Let $m\geq2$ and $f_1,f_2,\dots,f_m\in\mathcal{S}(\mathbb{R})$. Consider the multilinear spherical maximal function defined by 
$$\mathcal{S}^{m}(f_1,f_2,\dots,f_m)(x):=\sup_{t>0}\Big|\int_{\mathbb{S}^{m-1}}\prod^{m}_{i=1}f_i(x-ty_i)~d\sigma_{m-1}(\vec{y})\Big|,$$
where $d\sigma_{m-1}(\vec{y})$ is the normalized surface measure on the sphere $\mathbb{S}^{m-1}$. The corresponding single scale averaging operator was studied by Oberlin \cite{Oberlin88} and Bak and  Shim \cite{BakShim}. Later, Shrivastava and  Shuin \cite{ShrivastavaShuin} proved a complete $L^{p_1}(\R)\times L^{p_2}(\R)\times \cdots L^{p_m}(\R)\rightarrow L^{p}(\R)$ boundedness for Banach indices satisfying $\frac{1}{p}\leq \sum^{m}_{i=1}\frac{1}{p_i}$, $1\leq p_i,p\leq\infty$. Very recently, Dosidis and Ramos investigated the boundedness of the maximal function $\mathcal{S}^{m}$. They proved the following,
\begin{theorem}[\cite{DosidisRamos}]\label{Dosidis}
	Let $m\geq2$, $1< p_i\leq\infty$ for $i=1,2,\dots,m$ and $\frac{1}{p}=\sum^{m}_{i=1}\frac{1}{p_i}$. Then there exists $C>0$ such that the following boundedness holds 
	\begin{eqnarray}\label{eq1}
	\Vert \mathcal{S}^{m}(f_1,f_2,\dots,f_m)\Vert_{L^{p}(\mathbb{R})}\leq C\prod^{m}_{i=1}\Vert f_i\Vert_{L^{p_i}(\mathbb{R})},
	\end{eqnarray} if and only if 
	\begin{enumerate}[(a)]
		\item $\frac{1}{p}=\sum^{m}_{i=1}\frac{1}{p_i}<m-1$,\label{DRa}
		\item for every $i=1,2,\dots,m$, $\sum^{m}_{j=1,j\neq i}\frac{1}{p_j}<m-\frac{3}{2}$,\label{DRb}
		\item $(\frac{1}{p_1},\frac{1}{p_2},\dots,\frac{1}{p_m})\notin \{0,1\}^{m}\setminus\{(0,0,\dots,0)\}$.
	\end{enumerate}
	If $(\frac{1}{p_1},\frac{1}{p_2},\dots,\frac{1}{p_m})\in \{0,1\}^{m}\setminus\{(0,0,\dots,0)\}$, then weak type estimate holds, i.e. 
	$$\Vert \mathcal{S}^{m}(f_1,f_2,\dots,f_m)\Vert_{L^{p,\infty}(\mathbb{R})}\leq C\prod^{m}_{i=1}\Vert f_i\Vert_{L^{p_i}(\mathbb{R})}$$  if and only if \ref{DRa} and \ref{DRb} both hold. Moreover, if for some $i\in\{1,2,\dots,m\}$, $\sum^{m}_{j=1,j\neq i}\frac{1}{p_j}=m-\frac{3}{2}$, then the estimate $\mathcal S^m:L^{p_1}(\R)\times\cdots\times L^{p_m}(\R)\to L^{p,\infty}(\R)$ does not hold.
\end{theorem}
Our next result concerns the restricted weak-type estimates for $\mathcal S^m$ on the boundary points of the convex hull defined in Theorem \ref{Dosidis}. Namely, we have
\begin{coro}\label{mlinearspherical}
	Let $m\geq2$, $1\leq p_i\leq\infty$ for $i=1,2,\dots,m$ and $\frac{1}{p}=\sum^{m}_{i=1}\frac{1}{p_i}$. Then there exists $C>0$ such that 
	\begin{eqnarray}\label{eq1}
	\Vert \mathcal{S}^{m}(f_1,f_2,\dots,f_m)\Vert_{L^{p,\infty}(\mathbb{R})}\leq C\prod^{m}_{i=1}\Vert f_i\Vert_{L^{p_i,1}(\mathbb{R})}
	\end{eqnarray}
	holds true for $(\frac{1}{p_1},\frac{1}{p_2},\dots,\frac{1}{p_m})$ belonging to the following closed line segments 
	$$L_{k,j}=\left\{\left(\frac{1}{p_1},\frac{1}{p_2},\dots,\frac{1}{p_m}\right):0\leq \frac{1}{p_k}\leq\frac{1}{2}, \frac{1}{p_j}=\frac{1}{2},\frac{1}{p_i}=1,~\forall~i\neq k,j\right\},~~\forall~j,k\in\{1,2,\dots,m\}.$$
	%\begin{enumerate}
	%\item $\frac{1}{p}=\sum^{m}_{i=1}\frac{1}{p_i}=m-1$,
	%\item for every $i=1,2,\dots,m$, $\sum^{m}_{j=1,j\neq i}\frac{1}{p_j}=m-\frac{3}{2}$.
	%\item $(\frac{1}{p_1},\frac{1}{p_2},\dots,\frac{1}{p_m})\notin \{0,1\}^{m}\setminus\{(0,0,\dots,0)\}$.
	%\end{enumerate}
\end{coro}
\subsubsection{Local bilinear spherical maximal function}
In \cite{MaximalEstimatesForTheBilinearSphericalAveragesAndTheBilinearBochnerRieszOperators}, Jeong and Lee also studied the improving estimates for the local bilinear maximal operator,
$$\mathfrak{M}_{loc}(f,g)(x):=\sup_{1<t<2}\left|\int_{\mathbb S^{2d-1}}f(x-ty_1)g(x-ty_2)\;d\sigma(y_1,y_2)\right|.$$
They proved the following,
\begin{theorem}[\cite{MaximalEstimatesForTheBilinearSphericalAveragesAndTheBilinearBochnerRieszOperators}]\label{JeongLee}
	Let $d\geq2$, $1< p_1,p_2\leq\infty$, and $\frac{1}{2}<p<\infty$. Then the estimate
	\begin{equation}\label{localbounds}
		\|\mathfrak{M}_{loc}\|_{L^{p_1}(\R^d)\times L^{p_2}(\R^d)\to L^{p}(\R^d)}\lesssim 1,
	\end{equation}
	holds for $\frac{1}{p}\leq\frac{1}{p_1}+\frac{1}{p_2}<\min\left\{\frac{2d-1}{d},1+\frac{d}{p},\frac{1}{p}+\frac{2(d-1)}{d}\right\}$. Conversely \eqref{localbounds} holds only if $\frac{1}{p}\leq\frac{1}{p_1}+\frac{1}{p_2}\leq\min\left\{\frac{2d-1}{d},1+\frac{d}{p}\right\}$. Furthermore, for $p=\infty$, \eqref{localbounds} holds if and only if $0\leq\frac{1}{p_1}+\frac{1}{p_2}\leq1$.
\end{theorem}
\vspace{0.3cm}
The above range is sharp for the strong type boundedness for $0<p<d$ and $\frac{2(d-1)}{d-2}<p<\infty$. Using a Knapp type example, the authors in \cite{BFOPZ} showed that the condition $\frac{1}{p_1}+\frac{1}{p_2}\leq\frac{d-1}{(d+1)p}+\frac{2d}{d+1}$ is necessary for the strong $(p_1,p_2,p)-$boundedness of $\mathfrak{M}_{loc}$; however, the sufficiency of the condition was left open. Our next result provides new bounds for $\mathfrak{M}_{loc}$, thus improving the known results in all dimensions $d\geq2$. 
\iffalse{ 
	To state our results, we require a few notations. We define the line segments $\ell^d_1,\ell^d_2,$,$\ell^d_3$ and quadrilateral $\mathscr{Q}^d$ as follows:
\begin{align*}
	\ell^d_1&=\left\{\left(x,y,z\right)\in [0,1]^2\times[0,2):z=x+y=\frac{2d-1}{d}\right\},\\
	\ell^d_2&=\left\{\left(x,y,z\right)\in [0,1]^2\times[0,2):x+y=\frac{2d-1}{d}=\frac{2d-1}{2d+1}z+\frac{2(2d-1)}{2d+1}\right\},\\
	\ell^d_3&=\left\{\left(x,y,z\right)\in [0,1]^2\times[0,2):x+y=\frac{2d-1}{2d+1}z+\frac{2(2d-1)}{2d+1}=1+dz\right\},\\
	\mathscr{Q}^d&=\left\{\left(x,y,z\right)\in [0,1]^2\times[0,2):z<x+y=\frac{2d-1}{d}<\frac{2d-1}{2d+1}z+\frac{2(2d-1)}{2d+1}\right\}.
\end{align*}
}
\begin{figure}[H]
	\centering
	\includegraphics[width=0.5\textwidth]{Figure.png}
	\caption{The figure denotes the range of $L^{p_1}(\R^d)\times L^{p_2}(\R^d)\to L^p(\R^d)$ boundedness of $\mathfrak{M}_{loc}$ for $d\geq3$. The red region (convex hull generated by $C_1,C_2,E_1,E_2$, and $F$) is the additional range obtained in Theorem \ref{slicedbilinearimproving}.} 
	\label{Fig:slicedbilinearimproving}
\end{figure}
\fi
\begin{theorem}\label{slicedbilinearimproving}
	Let $d\geq2$, $1< p_1,p_2\leq\infty$, and $\frac{1}{2}<p<\infty$. Then the estimate \eqref{localbounds} holds if
	\begin{enumerate}
		\item $\frac{1}{p}\leq\frac{1}{p_1}+\frac{1}{p_2}<\min\left\{\frac{2d-1}{d},1+\frac{d}{p},\frac{2d-1}{(2d+1)p}+\frac{2(2d-1)}{2d+1}\right\}$.
		\item $\frac{2d^2+d+4}{p_1}+\frac{2d^2+d}{p_2}<\frac{2d^2+d}{p}+4d^2-2d+2$.
		\item $\frac{2d^2+d}{p_1}+\frac{2d^2+d+4}{p_2}<\frac{2d^2+d}{p}+4d^2-2d+2$.
	\end{enumerate}
\end{theorem}
\vspace{0.3cm}
\begin{remark}
	The range obtained in the above theorem is an improvement over the previously known $L^p-$boundedness range for the local bilinear spherical maximal function by Jeong and Lee (Theorem \ref{JeongLee}). More precisely, the additional range in Theorem \ref{slicedbilinearimproving} on which the boundedness holds is the open convex hull generated by the points $C_1=\left(1,\frac{d-1}{d},\frac{1}{d}\right),\;C_2=\left(\frac{d-1}{d},1,\frac{1}{d}\right),\;E_1=\left(1,\frac{d-2}{d-1},\frac{d-2}{d(d-1)}\right),\;E_2=\left(\frac{d-2}{d-1},1,\frac{d-2}{d(d-1)}\right),$ and $F=\left(\frac{(2d-1)^2}{2(2d^2-d+1)},\frac{(2d-1)^2}{2(2d^2-d+1)},\frac{2d-3}{2d^2-d+1}\right)$. We also note that the plane \[(2d^2+d+4)x+(2d^2+d)y=(2d^2+d)z+4d^2-2d+2\] passes through the points $C_1,E_1,$ and $F$. Similarly, the points $C_2,E_2,$ and $F$ lie on the plane $(2d^2+d+4)y+(2d^2+d)x=(2d^2+d)z+4d^2-2d+2$. We would like to mention that while the points $C_1,C_2,E_1$ and $E_2$ were implicit in Theorem \ref{JeongLee} of Jeong and Lee, the boundary point $F$ is the novel feature of our Theorem \ref{slicedbilinearimproving}.
\end{remark}
To prove Theorem  \ref{fullmaximal}\ref{bilinearone} and \ref{slicedbilinearimproving}, we will rely on a modification of slicing argument in \cite{MaximalEstimatesForTheBilinearSphericalAveragesAndTheBilinearBochnerRieszOperators}. In contrast to domination of $\mathfrak{M}$ by the product of Hardy-Littlewood and spherical maximal function, we will dominate $\mathfrak{M}$ by certain intermediate averaging operators. For $1\leq r\leq\infty$, we define
\begin{align*}
	\mathfrak{A}^rf(x)&=\|A_tf(x)\|_{L^r([1,2],t^{d-1}dt)}.
\end{align*}
Observe that $\mathfrak{A}^1$ and $\mathfrak{A}^\infty$ are the local Hardy-Littlewood and local spherical maximal functions respectively. We also define the maximal operator $\mathfrak{A}^r_*$ as follows,
\[\mathfrak{A}^r_*f(x)=\sup_{k\in\Z}\|A_{2^kt}f(x)\|_{L^r([1,2],t^{d-1}dt)},\;1\leq r\leq\infty.\]
\sloppy We set $A=\left(\frac{1}{r},0\right),\;P=\left(\frac{d+1+r(d^2-d)}{r(d^2+1)},\frac{d-1}{r'(d^2+1)}\right),\;Q=\left(\frac{rd-r+1}{rd},\frac{1}{r'd}\right),$ and $R=\left(\frac{rd-r+1}{rd},\frac{rd-r+1}{rd}\right)$. We denote $QR$ to be the open line segment joining the points $Q$ and $R$. We have the following boundedness for $\mathfrak{A}^r$ and $\mathfrak{A}^r_*$,
\begin{figure}[H]
	\centering
	\begin{tikzpicture}[scale=2.8]
		\tiny
		\fill[lightgray] (0,0)--(0.5,0)--(0.7,0.1)--(0.8333,0.3)--(0.8333,0.8333)--cycle;
		\draw[thin][->]  (0,0)node[left]{$O$} --(1.15,0) node[right]{$\frac{1}{p}$};
		\draw[thin][->]  (0,0) --(0,1.2) node[left]{$\frac{1}{q}$};
		\draw[densely dotted] (0,1)node[left]{$(0,1)$}--(1,1)node[right]{$(1,1)$};
		\draw [densely dotted] (1,0)node[below]{$(1,0)$} --(1,1);
		\draw[densely dotted] (0.5,0)node[below]{$A$}--(0.7,0.1)node[right]{$P$}--(0.8333,0.3)node[right]{$Q$}--(0.8333,0.8333)node[right]{$R$}--(0,0);		
	\end{tikzpicture}
	\caption{The figure denotes the region of $L^p(\R^d)\to L^q(\R^d)$ boundedness of $\mathfrak{A}^r$ when $\left(\frac{1}{p},\frac{1}{q}\right)$ is contained in the closed convex hull generated by the points $O,A,P,Q,R$ excluding the points $P,Q,R$.}
	\label{Fig:linearAr}
\end{figure}
\begin{theorem}\label{linearAr}
	Let $d\geq2$ and $1 \leq p,q\leq \infty$. The following holds true,
	\begin{enumerate}[(A)]
		\item For $1\leq r\leq\infty$ and $\frac{1}{q}\leq\frac{1}{p}\leq\min\{\frac{d}{q}+\frac{1}{r},\frac{d-1}{d}+\frac{1}{dr},\frac{d-1}{(d+1)q}+\frac{2}{(d+1)r}+\frac{d-1}{d+1}\}$ and $\left(\frac{1}{p},\frac{1}{q}\right)\notin\{P,Q,R\}\cup QR$, we have
		\begin{equation}\label{localA}
			\|\mathfrak{A}^r\|_{L^p(\R^d)\to{L^q(\R^d)}}\lesssim 1.
		\end{equation}\label{stronglinearA}
		Moreover for $\left(\frac{1}{p},\frac{1}{q}\right)\in\{P,Q,R\}$, we have the restricted weak type inequality,
			\[\|\mathfrak{A}^r\|_{L^{p,1}(\R^d)\to{L^{q,\infty}(\R^d)}}\lesssim 1,\]
		and for $\left(\frac{1}{p},\frac{1}{q}\right)\in QR$, we have the restricted strong type inequality
			\[\|\mathfrak{A}^r\|_{L^{p,1}(\R^d)\to{L^{q}(\R^d)}}\lesssim 1.\]
		\item Conversely, the estimate \eqref{localA} holds only if $\frac{1}{q}\leq\frac{1}{p}\leq\min\{\frac{d}{q}+\frac{1}{r},\frac{d-1}{d}+\frac{1}{dr},\frac{d-1}{(d+1)q}+\frac{2}{(d+1)r}+\frac{d-1}{d+1}\}$.\label{necessarylinearAr}
		\item For $1\leq r\leq\infty$, the operator $\mathfrak{A}^r_*$ maps $L^p(\R^d)$ to $L^p(\R^d)$ for $p>\frac{dr}{dr-r+1}$. Moreover $\mathfrak{A}^r_*$ is of restricted weak type $\left(\frac{dr}{dr-r+1},\frac{dr}{dr-r+1}\right)$ for $1\leq r<\infty$. \label{restrictedlinearA}
	\end{enumerate}
\end{theorem}
\begin{remark}\label{Rem:strongtype}
	It was pointed out in \cite{SeegerVariation} that the local spherical maximal function does not map $L^p(\R^d)$ to $L^q(\R^d)$ when $d\geq 3$ and $\left(\frac{1}{p},\frac{1}{q}\right)$ lies in the line segment joining the points $\left(\frac{d-1}{d},\frac{d-1}{d}\right)$ and $\left(\frac{d-1}{d},\frac{1}{d}\right)$. An analogous result also holds for the averaging operators $\mathfrak A^r$ when $d\geq 2$ i.e. a strong type boundedness does not hold for $\mathfrak A^r$ when $\left(\frac{1}{p},\frac{1}{q}\right)$ lies in the line segment $QR$, this is a corollary of Proposition \ref{restrictedexample}.
\end{remark}
For $1\leq r\leq\infty$ and $0<\delta<1$, we also consider a variant of the local maximal operator as follows,
\begin{align*}
	\mathfrak{B}^r_\delta f(x)&=\sup\limits_{1<t<2}\;\left(\frac{1}{\delta}\int_{1-\delta}^{1+\delta}|A_{ts}f(x)|^r s^{d-1}ds\right)^\frac{1}{r}.
\end{align*}
The sharp $L^p-$boundedness of $\mathfrak{B}^1_\delta$ in terms of $\delta$ was studied by Schlag \cite{Schlag} in dimension two. He proved the following,
\begin{theorem}[\cite{Schlag}]
	Let $0<\delta<1$. The following inequality holds true
	\[\|\mathfrak{B}^1_\delta\|_{L^p\to L^q}\lesssim A_\delta,\]
	whenever,
	\begin{enumerate}
		\item $A_\delta=1$ and $\left(\frac{1}{p},\frac{1}{q}\right)$ lies in the open convex hull generated by $(0,0),\;\left(\frac{2}{5},\frac{1}{5}\right),\;\left(\frac{1}{2},\frac{1}{2}\right)$.
		\item $A_\delta=C_\epsilon\delta^{\frac{2}{q}-\frac{1}{p}-\epsilon}$, $\epsilon>0$ and $\left(\frac{1}{p},\frac{1}{q}\right)$ lies in the closed convex hull generated by $(0,0),\;(1,0),\;\left(\frac{2}{5},\frac{1}{5}\right)$.
		\item $A_\delta=C_\epsilon\delta^{\frac{1}{2}(1+\frac{1}{q}-\frac{3}{p})-\epsilon}$, $\epsilon>0$ and $\left(\frac{1}{p},\frac{1}{q}\right)$ lies in the closed convex hull generated by $\left(\frac{2}{5},\frac{1}{5}\right),\;(1,0),\;\left(\frac{1}{2},\frac{1}{2}\right)$.
		\item $A_\delta=\delta^{1-\frac{2}{p}}$ and $\left(\frac{1}{p},\frac{1}{q}\right)$ lies in the closed convex hull generated by $\left(\frac{1}{2},\frac{1}{2}\right),\;(1,0),\;(1,1)$.
	\end{enumerate}
\end{theorem}
We note that the dependence on $\epsilon$ can be removed by using the estimates in \cite{Lee1}.

Our next result concerns the sharp boundedness of the operators $\mathfrak{B}^r_\delta$ for $1\leq r\leq\infty$ in all dimensions. We note that the range of $L^p-$boundedness of $\mathfrak{B}^r_\delta$ is same as that of $\mathfrak{A}^r$. However, we capture the sharp dependence of $L^p-$bounds of $\mathfrak{B}^r_\delta$ in terms of the constant $\delta$. We recall the points $O,A,P,Q,$ and $R$ defined before Theorem \ref{linearAr} and set $P'=\left(\frac{d^2-d}{d^2+1},\frac{d-1}{d^2+1}\right),\;Q'=\left(\frac{d-1}{d},\frac{1}{d}\right),$ and $R'=\left(\frac{d-1}{d},\frac{d-1}{d}\right)$. We have the following,

\begin{theorem}\label{linearBr}
	Let $1\leq r\leq\infty$ and $0<\delta<1$. The following inequality holds true
	\[\|\mathfrak{B}^r_\delta\|_{L^p\to L^q}\lesssim A_\delta,\]
	whenever,
	\begin{enumerate}[(i)]
		\item $A_\delta=1$ and $\left(\frac{1}{p},\frac{1}{q}\right)$ lies in the open convex hull generated by $O,P',Q',$ and $R'$.\label{linearBr1}
		\item $A_\delta=C_\epsilon\delta^{\frac{d}{q}-\frac{1}{p}}$ and $\left(\frac{1}{p},\frac{1}{q}\right)$ lies in the closed convex hull generated by $O,A,P,$ and $P'$ except the point $P$.\label{linearBr2}
		\item $A_\delta=C_\epsilon\delta^{\frac{1}{2}(d-1+\frac{d-1}{q}-\frac{d+1}{p})}$, and $\left(\frac{1}{p},\frac{1}{q}\right)$ lies in the closed convex hull generated by $P,Q,Q',$ and $P'$ except the points $P$ and $Q$.\label{linearBr3}
		\item $A_\delta=\delta^{d-1-\frac{d}{p}}$ and $\left(\frac{1}{p},\frac{1}{q}\right)$ lies in the closed convex hull generated by $Q,R,R'$ and $Q'$ except the line segment joining the points $Q$ and $R$.\label{linearBr4}
	\end{enumerate}
\end{theorem}
\begin{figure}[H]
	\centering
	\begin{tikzpicture}[scale=5]
		\tiny
		%\fill[lightgray] (0,0)--(0.5,0)--(0.7,0.1)--(0.8333,0.3)--(0.8333,0.8333)--cycle;
		\draw[thin][->]  (0,0)node[left]{$O$} --(1.15,0) node[right]{$\frac{1}{p}$};
		\draw[thin][->]  (0,0) --(0,1.2) node[left]{$\frac{1}{q}$};
		\draw[densely dotted] (0,1)node[left]{$(0,1)$}--(1,1)node[right]{$(1,1)$};
		\draw [densely dotted] (1,0)node[below]{$(1,0)$} --(1,1);
		\draw[densely dotted] (0.6,0.2)node[left]{$P'$}--(0.6667,0.3333)node[left]{$Q'$}--(0.6667,0.6667)node[above]{$R'$}--(0,0)--cycle;
		\draw[thin] (0.25,0)node[below]{$A$}--(0.7,0.15)node[right]{$P$}--(0.75,0.25)node[right]{$Q$}--(0.75,0.75)node[right]{$R$}--(0,0);
		\draw[densely dotted] (0.6,0.2)--(0.7,0.15);
		\draw[densely dotted] (0.6667,0.3333)--(0.75,0.25);
		\draw[densely dotted] (0.45,0.3)node{$\ref{linearBr1}$};
		\draw[densely dotted] (0.45,0.1)node{$\ref{linearBr2}$};
		\draw[densely dotted] (0.675,0.23)node{$\ref{linearBr3}$};
		\draw[densely dotted] (0.71,0.5)node{$\ref{linearBr4}$};
	\end{tikzpicture}
	\caption{The figure denotes the regions of $L^p(\R^d)\to L^q(\R^d)$ boundedness of $\mathfrak{B}^r_\delta$ mentioned in Theorem \ref{linearBr}.}
	\label{Fig:linearAr}
\end{figure}
The proof of Theorem \ref{linearBr} is similar to that of Theorem \ref{linearAr} where we interpolate the estimates for the cases $r=1$ and $r=\infty$. We skip the details.

We prove Theorem \ref{fullmaximal}, Corollary \ref{mlinearspherical} and Theorem \ref{slicedbilinearimproving} in Section \ref{Sec:fullmaximal}. The proof of Theorem \ref{linearAr} is contained in Section \ref{Sec:localmaximal}.

\subsection{Part II}
For $0<\theta<2\pi$, let $\Theta$ be the counter-clockwise rotation by the angle $\theta$. Consider the bilinear average given by 
\[\mathcal A_t^\theta (f,g)(x)=\int_{\s^{1}}f(x-ty)g(x-t\Theta y)\;d\sigma(y),\;t>0.\]
When $\theta=\frac{\pi}{3}$, the operator $\mathcal{A}^{\theta}_1$ corresponds to the triangle averaging operator on the plane. Such operator first appeared in the work of Greenleaf and Iosevich \cite{GI} in the context of Falconer type problem for triangles. Later, Palsson and Sovine investigated the $L^{p_1}\times L^{p_2} \to L^p-$boundedness of the operator $\mathcal{A}_1^{\frac{\pi}{3}}$ in \cite{PS2020}. We also refer to \cite{CLM2021} and \cite{IPS} for boundedness results concerning mutltilinear averaging operators associated to simplices. Greenleaf et. al. \cite{GIKL} proved the following boundedness result for $\mathcal{A}^{\theta}_1$ for $0<\theta<2\pi$.
\begin{theorem}[\cite{GIKL}]
	Let $0<\theta<2\pi$ and $1\leq p_1,p_2,p\leq\infty$. The operator $\mathcal A^\theta_1$ is bounded from $L^{p_1}(\R^2)\times L^{p_2}(\R^2)$ to $L^p(\R^2)$ if
	\begin{enumerate}
		\item $\theta\neq\pi$ and $\left(\frac{1}{p_1},\frac{1}{p_2},\frac{1}{p}\right)$ lies in the closed convex hull generated by the vertices $(0,0,0),\;\left(\frac{2}{3},\frac{2}{3},1\right),\;\left(0,\frac{2}{3},\frac{1}{3}\right),\;\left(\frac{2}{3},0,\frac{1}{3}\right),\;(1,0,1),\;(0,1,1),$ and $\left(\frac{1}{2},\frac{1}{2},\frac{1}{2}\right)$, or
		\item $\theta=\pi$ and $\left(\frac{1}{p_1},\frac{1}{p_2},\frac{1}{p}\right)$ lies in the closed convex hull generated by the vertices $(0,0,0),\;\left(\frac{2}{3},\frac{2}{3},1\right),\;\left(0,\frac{2}{3},\frac{1}{3}\right),\;\left(\frac{2}{3},0,\frac{1}{3}\right),\;(1,0,1),$ and $(0,1,1)$.
	\end{enumerate}
\end{theorem}
The boundedness of $\mathcal A_1^\theta$ is sharp in the H\"older range of indices mentioned in the above theorem, i.e. when $p\geq 1$. However, it is unknown if the boundedness holds outside the closed convex hull in the above theorem for $p<1$. We also refer to \cite{CLS} for related results for $\mathcal A_t^\theta$.
\begin{remark}\label{Rem:necessary}
	In \cite{GIKL}, the authors showed that 
	\begin{equation}\label{necessarycondition}
		\frac{3}{p_1}+\frac{3}{p_2}\leq 1+\frac{3}{p},
	\end{equation} 
	is a necessary condition for the operator $\mathcal{A}^\pi_1$ to be bounded from $L^{p_1}(\R^2)\times L^{p_2}(\R^2)$ to $L^{p}(\R^2)$. In Section \ref{Sec:Necessary}, we will provide a different example based on functions of product type, that works for dimensions $d\geq2$. In particular, we show that the condition \eqref{necessarycondition} is in fact necessary for boundedness of $\mathcal{A}^\pi_1$ even when the functions are restricted to the space of functions of product type. We refer to \cite{Tanaka} for analogous results for the Kakeya maximal function acting on functions of product type. 
        
    \end{remark}

In this article, we are concerned with the study of the full maximal function given by,
\[\mathcal M^\theta (f,g)(x)=\sup_{t>0}\mathcal |A_t^\theta(f,g)(x)|.\]
We note that $\mathcal M^\theta$ is bounded from $L^{p_1}(\R^2)\times L^{p_2}(\R^2)$ to $L^p(\R^2)$ for $2<p\leq\infty$ with $\frac{1}{p_1}+\frac{1}{p_2}=\frac{1}{p}$. Indeed, the boundedness is a consequence of bilinear interpolation, the linear estimates $A_*:L^p(\R^2)\to L^p(\R^2),\;p>2$ and the pointwise inequality
\[\mathcal M^\theta(f,g)(x)\leq\min\{\|f\|_\infty A_*g(x),\|g\|_\infty A_*f(x)\}.\] 
We note that Palsson and Sovine \cite{PS2020} had observed the above pointwise estimates and related boundedness results for the maximal triangle averaging operator $\mathcal{M}^{\frac{\pi}{3}}$.

Our first main result concerns the restricted weak type non-inequality of the maximal function $\mathcal M^\theta$ at the endpoint boundaries $p_1=2$ and $p_2=2$ in dimension two. We have,
\begin{theorem}\label{Mfull}
	Let $1<p_1,p_2\leq\infty,\;\frac{1}{p_1}+\frac{1}{p_2}=\frac{1}{p}$ and $0<\theta<2\pi$. The operator $\mathcal{M}^\theta$ does not map $L^{p_1,1}(\R^2)\times L^{p_2,1}(\R^2)$ to $L^{p,\infty}(\R^2)$ if $p_1\leq2$ or $p_2\leq2$. In particular, $\mathcal{M}^\theta$ is not of restricted weak type $(2,2,1)$.
\end{theorem}
Currently, we do not have a positive result for the operator $\mathcal M^\theta$ in the local $L^2$ range: $p_1,p_2>2$ and $1<p\leq2$. However, motivated by the Oberlin's result \cite{Oberlin88} for $\widetilde{A}$, we will study the boundedness of a bilinear operator defined by taking a fixed choice of radii $(t(x)=|x|)$ in contrast to the definition of $\mathcal{M}^\theta$. More precisely, we will prove sharp boundedness results for the bilinear operator $\widetilde{\mathcal A}^\theta$ defined as
\[\widetilde{\mathcal A}^\theta (f,g)(x)=\int_{\s^{1}}f(x-|x|y)g(x-|x|\Theta y)\;d\sigma(y).\]
We have the following.
\begin{theorem}\label{linearized}
	Let $2<p_1,p_2\leq\infty$ with $\frac{1}{p_1}+\frac{1}{p_2}=\frac{1}{p}$. Then 
	$$\Vert\widetilde{\mathcal A}^\theta(f,g)\Vert_{L^{p}(\mathbb{R}^{2})}\lesssim \Vert f\Vert_{L^{p_1}(\mathbb{R}^{2})}\Vert g\Vert_{L^{p_2}(\mathbb{R}^{2})}.$$
	Moreover, for $p_1=2$ or $p_2=2$ we have the following restricted weak type estimates,
	$$\Vert\widetilde{\mathcal A}^\theta(f,g)\Vert_{L^{p,\infty}(\mathbb{R}^{2})}\lesssim \Vert f\Vert_{L^{p_1,1}(\mathbb{R}^{2})}\Vert g\Vert_{L^{p_2,1}(\mathbb{R}^{2})}.$$
\end{theorem}
The proof is based on an appropriate change of variable in the polar co-ordinates and a multilinear version of Bourgain's interpolation trick (Lemma \ref{interpolation}).
\begin{remark}
	We would like to remark that using our method of proof, one can recover the result of Oberlin \cite{OberlinLinearization} that the corresponding linear spherical operator $\widetilde A$ is also of restricted weak type $(2,2)$. This simplifies the proof of Oberlin \cite{OberlinLinearization}. 
\end{remark}
The proofs of Theorem \ref{Mfull} and Theorem \ref{linearized} are contained in Section \ref{Mfullsec}. In Section \ref{Sec:Necessary}, we also discuss some necessary conditions for $L^p-$ boundedeness of higher dimensional version of $\mathcal{M}^\theta$.
\section{{Bilinear spherical maximal function $\mathfrak{M}$: Proofs of Theorem \ref{fullmaximal}, Corollary \ref{mlinearspherical}, and Theorem \ref{slicedbilinearimproving}}}\label{Sec:fullmaximal}
\subsection{Proof of Theorem \ref{fullmaximal}\ref{bilinearone}:}
We note that by interpolation, it is enough to show the following three bounds:
\begin{align*}
	\mathfrak{M}&:L^{2,1}(\R)\times L^{2,1}(\R)\to L^{1,\infty}(\R),\\
	\mathfrak{M}&:L^{2,1}(\R)\times L^{\infty}(\R)\to L^{2,\infty}(\R),\\
	\mathfrak{M}&:L^{\infty}(\R)\times L^{2,1}(\R)\to L^{2,\infty}(\R).
\end{align*} 
We first show that $\mathfrak{M}:L^{2,1}(\R)\times L^{2,1}(\R)\to L^{1,\infty}(\R)$. Observe that due to symmetry, it is enough to deal with the integral over the arc from $\theta=0$ to $\theta=\frac{\pi}{4}$ instead of integral over $\mathbb{S}^{1}$. By a change of variable, we have
\begin{align*}
	\mathfrak{M}(f,g)(x)\lesssim\sup_{t>0}\int_{y=0}^{\frac{1}{\sqrt{2}}}f(x-ty)g(x-t\sqrt{1-y^2})\;\frac{dy}{\sqrt{1-y^2}}+\text{ similar terms}.
\end{align*}
By decomposing the interval $[0,\frac{1}{\sqrt{2}}]$ into dyadic annuli, we have
\[\mathfrak{M}(f,g)(x)\lesssim\sum_{k=1}^{\infty}T_k(f,g)(x),\]
where the operator $T_k$ is defined by
\begin{align*}
	T_k(f,g)(x):&=\sup_{t>0}\int_{2^{-k-1/2}}^{2^{-k+1/2}}|f(x-ty)||g(x-t\sqrt{1-y^2})|\;dy
\end{align*}
We have the following pointwise inequality,
\begin{equation}\label{T_k}
	T_k(f,g)(x)\lesssim\min\;\left\{2^{\frac{k}{3}}M_3f(x)M_{\frac{3}{2}}g(x),2^{-\frac{k}{3}}M_\frac{3}{2}f(x)M_3g(x)\right\},
\end{equation}
where $M_sf(x)=\Big(M\big(|f|^s\big)(x)\Big)^{\frac{1}{s}}$ for $1\leq s<\infty$ and $M$ denotes the Hardy--Littlewood maximal function.

Indeed, by H\"older's inequality, we get
\begin{align*}
	&T_k(f,g)(x)\\
	\lesssim&\sup_{t>0}\bigg(\int_{2^{-k-1/2}}^{2^{-k+1/2}}|f(x-ty)|^3\;dy\bigg)^\frac{1}{3}\;\sup_{t>0}\bigg(\int_{2^{-k-1/2}}^{2^{-k+1/2}}|g(x-t\sqrt{1-y^2})|^\frac{3}{2}\;dy\bigg)^\frac{2}{3}\\
	=&2^{-\frac{k}{3}}\sup_{t>0}\bigg(\int_{\frac{1}{\sqrt{2}}}^{\sqrt{2}}|f(x-2^{-k}ty)|^3\;dy\bigg)^\frac{1}{3}\;\sup_{t>0}\bigg(\int_{\sqrt{1-2^{-2k+1}}}^{\sqrt{1-2^{-2k-1}}}|g(x-tz)|^\frac{3}{2}\;\frac{zdz}{\sqrt{1-z^2}}\bigg)^\frac{2}{3}\\
	\lesssim& 2^{-\frac{k}{3}}\sup_{t>0}\bigg(\int_{\frac{1}{\sqrt{2}}}^{\sqrt{2}}|f(x-2^{-k}ty)|^3\;dy\bigg)^\frac{1}{3}\;\sup_{t>0}\bigg(\int_{\sqrt{1-2^{-2k+1}}}^{\sqrt{1-2^{-2k-1}}}|g(x-tz)|^\frac{3}{2}\;\frac{dz}{2^{-k}}\bigg)^\frac{2}{3}\\
	\lesssim& 2^{\frac{k}{3}}\sup_{t>0}\bigg(\int_{0}^{\sqrt{2}}|f(x-(2^{-k}t)y)|^3\;dy\bigg)^\frac{1}{3}\;\sup_{t>0}\bigg(\int_{0}^{1}|g(x-tz)|^\frac{3}{2}\;dz\bigg)^\frac{2}{3}\\
	\lesssim& 2^{\frac{k}{3}}M_3f(x)M_{\frac{3}{2}}g(x),
\end{align*}
where we applied a change of variable $y\to 2^{-k}y$ for the first integral and $\sqrt{1-y^2}\to z$ for the second integral in the second step. A similar argument will give us the other inequality in \eqref{T_k}.

Therefore, for a fixed $N\in\N$, we have
\begin{align*}
	\mathfrak{M}(f,g)(x)&\lesssim\sum_{k=1}^N2^{\frac{k}{3}}M_3f(x)M_{\frac{3}{2}}g(x)+\sum_{k=N+1}^\infty2^{-\frac{k}{3}}M_\frac{3}{2}f(x)M_3g(x)\\
	&\lesssim 2^{\frac{N}{3}}M_3f(x)M_{\frac{3}{2}}g(x)+2^{-\frac{N}{3}}M_\frac{3}{2}f(x)M_3g(x).
\end{align*}
Hence using the weak type bounds $M_s:L^s\to L^{s,\infty},\;s\geq 1,$ we obtain
\begin{align*}
	|\{x\in\R:\;\mathfrak{M}(\chi_F,\chi_G)(x)>\lambda\}|&\lesssim \frac{1}{\lambda}\left(2^{\frac{N}{3}}|F|^{\frac{1}{3}}|G|^\frac{2}{3}+2^{-\frac{N}{3}}|F|^{\frac{2}{3}}|G|^\frac{1}{3}\right)\\
	&=\frac{1}{\lambda}|F|^\frac{1}{2}|G|^\frac{1}{2},
\end{align*}
where we have chosen $N=3\log_2(|F|^{\frac{1}{6}}|G|^{-\frac{1}{6}})$. This proves that $\mathfrak{M}$ is bounded from $L^{2,1}(\R)\times L^{2,1}(\R)\to L^{1,\infty}(\R)$. 

Next, we show that $\mathfrak{M}$ is bounded from $L^{2,1}(\R)\times L^{\infty}(\R)\to L^{2,\infty}(\R)$. By symmetry, a similar argument will give us that $\mathfrak{M}$ is bounded from $L^{\infty}(\R)\times L^{2,1}(\R)\to L^{2,\infty}(\R)$.

Since $M_sf(x)\leq \|f\|_{L^\infty}$ for $1\leq s<\infty$, we get the following estimate from \eqref{T_k}.
\[T_k(f,g)(x)\lesssim\min\;\left\{2^{\frac{k}{3}}M_3f(x),2^{-\frac{k}{3}}M_\frac{3}{2}f(x)\right\}\|g\|_{L^\infty}.\]
Then, $\mathfrak{M}(f,g)(x)\lesssim\big[2^{\frac{N}{3}}M_3f(x)+2^{-\frac{N}{3}}M_\frac{3}{2}f(x)\big]\|g\|_{L^\infty}$ for any $N\in\N$, and
\begin{align*}
	|\{x\in\R:\;\mathfrak{M}(\chi_F,\chi_G)(x)>\lambda\}|&\lesssim \frac{2^N|F|}{\lambda^3}+\frac{2^{\frac{-N}{2}}|F|}{\lambda^{\frac{3}{2}}}.
\end{align*}
Choosing $N=\log_2(\lambda)$ proves the desired estimate.
\qed

\subsection{Proof of Corollary \ref{mlinearspherical}:}
In order to prove this theorem we invoke the slicing argument from \cite{MaximalEstimatesForTheBilinearSphericalAveragesAndTheBilinearBochnerRieszOperators}. Applying slicing argument we get for $m\geq3$,
\begin{eqnarray*}
	\mathcal{S}^{m}(f_1,f_2,\dots,f_m)(x)&=&\sup_{t>0}\Big|\int_{B^{m-2}(0,1)}\prod^{m-2}_{i=1}f_i(x-ty_i)\\
	&&\hspace{15mm}\int_{r_y\mathbb{S}^{1}}f_{m-1}(x-ty_{m-1})f_{m}(x-ty_{m})d\sigma_{r_y}d\vec{y}\Big|,
\end{eqnarray*}
where $r_y=\sqrt{1-|\tilde{y}|^{2}}$, $\tilde{y}=(y_1,y_2,\dots,y_{m-2})$ and $d\vec{y}=\prod^{m-2}_{i=1}dy_i$. Now, applying a change of variable we get 
\begin{eqnarray*}
	&&\mathcal{S}^{m}(f_1,f_2,\dots,f_m)(x)\\
	&&=\sup_{t>0}\Big|\int_{B^{m-2}(0,1)}\prod^{m-2}_{i=1}f_i(x-ty_i)\int_{\mathbb{S}^{1}}f_{m-1}(x-tr_yy_{m-1})f_{m}(x-tr_yy_{m})d\sigma d\vec{y}\Big|\\
	&&\lesssim \prod^{m-2}_{i=1}Mf_i(x)\mathfrak{M}(f_{m-1},f_m)(x).
\end{eqnarray*}
Here, $M$ denote the Hardy--Littlewood maximal function. Note that due to symmetry of the sphere $\mathbb{S}^{m-1}$ we can interchange the role of the functions $f_i$ for $i=1,2,\dots,m$ and deduce the following inequality 
\begin{eqnarray}
\mathcal{S}^{m}(f_1,f_2,\dots,f_m)(x)\lesssim \mathfrak{M}(f_{j},f_k)(x)\prod^{m}_{i=1, i\neq j,k}Mf_i(x),
\end{eqnarray}
for any $j,k\in \{1,2,\dots,m\}$. Therefore, using the estimates of Theorem \ref{fullmaximal}\ref{bilinearone} we get the desired restricted weak-type estimates.
\qed

\subsection{Proof of Theorem \ref{fullmaximal}\ref{bilineartwo}:}
We claim that the following holds,
\begin{equation}\label{domination}
	\mathfrak{M}(f,g)(x)\lesssim \mathfrak {A}^{r}_*(f)(x)\mathfrak {A}^{r'}_*(g)(x).
\end{equation}
The Theorem \ref{fullmaximal}\ref{bilineartwo} follows at once by the above claim along with H\"older's inequality and Theorem \ref{linearAr}\ref{restrictedlinearA}.
We prove the above claim. We will employ the slicing argument from \cite{MaximalEstimatesForTheBilinearSphericalAveragesAndTheBilinearBochnerRieszOperators} followed by a polar decomposition to rewrite the bilinear spherical average as an average of product of two spherical averages over $\mathbb S^1$. This was previously used in the work of Borges and Foster \cite{Borges} and an analogous strategy was also used by Anderson, Kumchev and Palsson \cite{AKP2024} in the discrete setting. 
\begin{eqnarray*}
	&&\mathfrak{M}(f,g)(x)\\
    &=&2\sup_{t>0}\left|\int_{B^2(0,1)} f(x-ty) \int_{\mathbb{S}^1} g(x-t\sqrt{1-|y|^2}y_2) d\sigma(y_2) dy\right|\\
	&=&2\sup_{t>0}\left|\int_{0}^{1}\int_{\mathbb S^{1}}f(x-tsy_1)\;d\sigma(y_1)\int_{\mathbb S^{1}}g(x-t\sqrt{1-s^2}y_2)\;d\sigma(y_2)s\;ds\right|\\
	&\leq&2 \sup_{t>0}\left(\int_{0}^{1}\left|A_{ts}f(x)\right|^rs\;ds\right)^{\frac{1}{r}}\left(\int_{0}^{1}\left|A_{t\sqrt{1-s^2}}g(x)\right|^{r'}s\;ds\right)^{\frac{1}{r'}}\\
	&\leq&2 \sup_{t>0}\left(\int_{0}^{1}\left|A_{ts}f(x)\right|^rs\;ds\right)^{\frac{1}{r}}\left(\int_{0}^{1}\left|A_{ts}g(x)\right|^{r'}s\;ds\right)^{\frac{1}{r'}},
\end{eqnarray*}
where we applied H\"older's inequality with exponents $\frac{1}{r}+\frac{1}{r'}=1$ in the second last step and a change of variable $\sqrt{1-s^2}\to s$ in the last step. Thus, it is enough to show that 
\[\sup_{t>0}\left(\int_0^1|A_{ts}f(x)|^r~s ds\right)^{\frac{1}{r}}\lesssim \mathfrak {A}^{r}_*(f)(x).\]
By a change of variable $s\to\frac{s}{2}$ in the first term of second line below, we have
\begin{align*}
	&\sup_{t>0}\left(\int_0^1|A_{ts}f(x)|^r~sds\right)^{\frac{1}{r}}\\
	\leq&\sup_{t>0}\left(\int_0^{\frac{1}{2}}|A_{ts}f(x)|^rsds\right)^{\frac{1}{r}}+\sup_{t>0}\left(\int_{\frac{1}{2}}^1|A_{ts}f(x)|^rsds\right)^{\frac{1}{r}}\\
	=&\frac{1}{2^{\frac{2}{r}}}\sup_{t>0}\left(\int_0^1|A_{\frac{ts}{2}}f(x)|^rsds\right)^{\frac{1}{r}}+\sup_{k\in\Z}\sup_{2^k\leq t\leq2^{k+1}}\left(\int_{\frac{1}{2}}^1|A_{ts}f(x)|^rsds\right)^{\frac{1}{r}}\\
    \lesssim&\frac{1}{2^{\frac{2}{r}}}\sup_{t>0}\left(\int_0^1|A_{ts}f(x)|^rs ds\right)^{\frac{1}{r}}+\sup_{k\in\Z}\sup_{2^k\leq t\leq2^{k+1}}\left(\int_{t2^{-k-1}}^{t2^{-k}}|A_{2^{k}s}f(x)|^r(2^{k}t^{-1})^2sds\right)^{\frac{1}{r}}\\
    \leq& \frac{1}{2^{\frac{2}{r}}}\sup_{t>0}\left(\int_0^1|A_{ts}f(x)|^rs ds\right)^{\frac{1}{r}}+\sup_{k\in\Z}\left(\int_{\frac{1}{2}}^2|A_{2^{k}s}f(x)|^rsds\right)^{\frac{1}{r}},
\end{align*}
where we have applied a change of variable $s\to 2^{k}t^{-1}s$ in the second last line and since $\frac{1}{2}\leq{t2^{-k-1}}\leq{t2^{-k}}\leq2$, the supremum over $t\in[2^k,2^{k+1}]$ is absorbed in the integral.
Thus, the claim follows.\qed

\subsection{Proof of Theorem \ref{slicedbilinearimproving}:} We employ the slicing argument followed by a polar decomposition as in the proof of Theorem \ref{fullmaximal}\ref{bilineartwo} to obtain
\begin{align*}
	&\mathfrak{M}_{loc}(f,g)(x)\\
    =&2\sup_{1\leq t\leq2} \left| \int_{B^{d}(0,1)} f(x-ty) (1-|y|^2)^{\frac{d-2}{2}} \int_{\mathbb S^{d-1}}g(x-t\sqrt{1-|y|^2}y_2)\;d\sigma(y_2) dy\right|\\
	\lesssim&\sup_{1\leq t\leq2}\left|\int_{0}^{\frac{1}{\sqrt{2}}}\int_{\mathbb S^{d-1}}f(x-tsy_1)\;d\sigma(y_1)\int_{\mathbb S^{d-1}}g(x-t\sqrt{1-s^2}y_2)\;d\sigma(y_2)s^{d-1}(1-s^2)^{\frac{d-2}{2}}\;ds\right|\\
	+&\sup_{1\leq t\leq2}\left|\int_{\frac{1}{\sqrt{2}}}^{1}\int_{\mathbb S^{d-1}}f(x-tsy_1)\;d\sigma(y_1)\int_{\mathbb S^{d-1}}g(x-t\sqrt{1-s^2}y_2)\;d\sigma(y_2)s^{d-1}(1-s^2)^{\frac{d-2}{2}}\;ds\right|\\
	=&I+II.
\end{align*}
By decomposing the interval $[0,\frac{1}{\sqrt{2}}]$ into dyadic intervals and H\"older's inequality, we obtain 
\begin{align*}
	I&=\sup_{1\leq t\leq2}\left|\sum_{k=1}^{\infty}\int_{2^{-k-\frac{1}{2}}}^{2^{-k+\frac{1}{2}}}\int_{\mathbb S^{d-1}}f(x-tsy_1)\;d\sigma(y_1)\int_{\mathbb S^{d-1}}g(x-t\sqrt{1-s^2}y_2)\;d\sigma(y_2)s^{d-1}(1-s^2)^{\frac{d-2}{2}}\;ds\right|\\
    &\leq \sum^{\infty}_{k=1} \sup_{1\leq t\leq2} \left|\int_{2^{-k-\frac{1}{2}}}^{2^{-k+\frac{1}{2}}} A_{ts}f(x) s^{\frac{d-1}{r}} (1-s^2)^{\frac{d-2}{2r}} A_{t\sqrt{1-s^2}}g(x) s^{\frac{d-1}{r'}} (1-s^2)^{\frac{d-2}{2r'}} ~ds\right|\\
	&\lesssim \sum^{\infty}_{k=1} \sup_{1\leq t\leq2} \left(\int^{2^{-k+\frac{1}{2}}}_{2^{-k-\frac{1}{2}}} |A_{ts}f(x)|^{r} s^{d-1} (1-s^2)^{\frac{d-2}{2}}ds\right)^{1/r} \\
	&\hspace{4cm}\left( \int^{2^{-k+\frac{1}{2}}}_{2^{-k-\frac{1}{2}}} |A_{t\sqrt{1-s^2}}g(x)|^{r'} s^{d-1} (1-s^2)^{\frac{d-2}{2}}ds\right)^{1/r'}\\
	&\lesssim \sum^{\infty}_{k=1} \sup_{1\leq t\leq2} \left(\int^{2^{-k+\frac{1}{2}}}_{2^{-k-\frac{1}{2}}} |A_{ts}f(x)|^{r} s^{d-1} ds\right)^{1/r} \left( \int^{\sqrt{1-2^{-2k-1}}}_{\sqrt{1-2^{-2k+1}}} |A_{ts}g(x)|^{r'} s^{d-1} (1-s^2)^{\frac{d-2}{2}}ds\right)^{1/r'}\\
    &\lesssim \sum^{\infty}_{k=1} \sup_{1\leq t\leq2} \left(\int^{2^{-k+\frac{1}{2}}}_{2^{-k-\frac{1}{2}}} |A_{ts}f(x)|^{r} s^{d-1} ds\right)^{1/r} \left( 2^{-k(d-2)}\int^{\sqrt{1-2^{-2k-1}}}_{\sqrt{1-2^{-2k+1}}} |A_{ts}g(x)|^{r'} s^{d-1} ds\right)^{1/r'}\\
	&\leq\sum_{k=1}^{\infty}2^{-\frac{kd}{r'}}\sup_{1\leq t\leq2}\left(\int_{t}^{2t}\left|A_{2^{-k-\frac{1}{2}}s}f(x)\right|^r(2^{-k-\frac{1}{2}}t^{-1})^ds^{d-1}ds\right)^{\frac{1}{r}}\left(\frac{1}{2^{-2k}}\int_{\sqrt{1-2^{-2k+1}}}^{\sqrt{1-2^{-2k-1}}}\left|A_{ts}g(x)\right|^{r'}s^{d-1}\;ds\right)^{\frac{1}{r'}}\\
	&\lesssim\sum_{k=1}^{\infty}2^{-kd}\left(\int_{1}^{4}\left|A_{2^{-k-\frac{1}{2}}s}f(x)\right|^rs^{d-1}ds\right)^{\frac{1}{r}}\sup_{1\leq t\leq2}\left(\frac{1}{2^{-2k}}\int_{\sqrt{1-2^{-2k+1}}}^{\sqrt{1-2^{-2k-1}}}\left|A_{ts}g(x)\right|^{r'}s^{d-1}\;ds\right)^{\frac{1}{r'}}\\
	&\lesssim \sum_{k=1}^{\infty}2^{-kd}\|A_t(f_{2^{-k-\frac{1}{2}}})(2^{k+\frac{1}{2}}x)\|_{L^r([1,4],t^{d-1}dt)}\mathfrak{B}^{r'}_{2^{-2k}}g(x),
\end{align*}
where we apply a change of variable $s\to 2^{-k-\frac{1}{2}}t^{-1}s$ in the first integral in the third last inequality. We also note that 
\[\|\|A_tf(\cdot)\|_{L^r([1,4],t^{d-1}dt)}\|_p\lesssim\|\mathfrak{A}^rf\|_p.\]
We note that in a similar manner we obtain
\[II\lesssim  \sum_{k=1}^{\infty}2^{-kd}\mathfrak{B}^{r}_{2^{-2k}}f(x)\|A_t(g_{2^{-k-\frac{1}{2}}})(2^{k+\frac{1}{2}}x)\|_{L^{r'}([1,4],t^{d-1}dt)}.\]
Let $1\leq q_1,q_2,r\leq\infty$ be such that $\frac{1}{p}=\frac{1}{q_1}+\frac{1}{q_2}$ satisfying,
\begin{align}
	\begin{split}\label{condition1}
	\frac{1}{q_1}\leq\frac{1}{p_1}&\leq\min\{\frac{d}{q_1}+\frac{1}{r},\frac{d-1}{d}+\frac{1}{dr},\frac{d-1}{(d+1)q_1}+\frac{2}{(d+1)r}+\frac{d-1}{d+1}\},\\
	\frac{1}{q_2}\leq\frac{1}{p_2}&\leq\min\{\frac{d}{q_2}+\frac{1}{r'},\frac{d-1}{d}+\frac{1}{dr'},\frac{d-1}{(d+1)q_2}+\frac{2}{(d+1)r'}+\frac{d-1}{d+1}\}.
	\end{split}
\end{align}
By using H\"older's inequality and Theorems \ref{linearAr} and \ref{linearBr}, we obtain,
\begin{align*}
	\|\mathfrak{M}_{loc}(f,g)\|_{p}&\lesssim\sum_{k=1}^\infty2^{-kd}\left(\|\mathfrak{A}^{r}(f_{2^{-k-\frac{1}{2}}})(2^{k+\frac{1}{2}}\cdot)\|_{q_1}\|\mathfrak{B}^{r'}_{2^{-2k}}g\|_{q_2}+\|\mathfrak{B}^{r}_{2^{-2k}}f\|_{q_1}\|\mathfrak{A}^{r'}g_{2^{-k-\frac{1}{2}}}(2^{k+\frac{1}{2}}\cdot)\|_{q_2}\right)\\
	&\lesssim\sum_{k=1}^\infty\left(2^{-k\left(2d-1+\frac{d}{q_1}+\frac{d-1}{q_2}-\frac{d}{p_1}-\frac{d+1}{p_2}\right)}+2^{-k\left(2d-1+\frac{d-1}{q_1}+\frac{d}{q_2}-\frac{d+1}{p_1}-\frac{d}{p_2}\right)}\right)\|f\|_{p_1}\|g\|_{p_2}.
\end{align*}
The above series is summable if 
\begin{align}
	\begin{split}\label{condition2}
	\frac{d}{p_1}+\frac{d+1}{p_2}&<2d-1+\frac{d}{q_1}+\frac{d-1}{q_2},\;\text{and}\\
	\frac{d+1}{p_1}+\frac{d}{p_2}&<2d-1+\frac{d-1}{q_1}+\frac{d}{q_2}.
	\end{split}
\end{align}

Consider the triangular region $\mathscr{R}_1$ given by, \[\mathscr{R}_1:=\left\{(x,x,z):\;x<\min\left\{\frac{d}{2}z+\frac{1}{2},\frac{2d-1}{2d},\frac{d-1}{2(d+1)}z+\frac{d}{d+1},\frac{2d-1}{2d+1}+\frac{2d-1}{2(2d+1)}z\right\}\right\}.\] It is not difficult to see that the set of indices $p_1,p_2,p$ such that $\left(\frac{1}{p_1},\frac{1}{p_2},\frac{1}{p}\right)\in \mathscr{R}_1$ satisfies the conditions \ref{condition1} and \ref{condition2} when $r=2$ and $q_1=q_2=2p$. Therefore we obtain $\mathfrak{M}_{loc}:L^{p_1}(\R^d)\times L^{p_2}(\R^d)\to L^p(\R^d)$ for the indices $p_1,p_2,p$ satisfying $\left(\frac{1}{p_1},\frac{1}{p_2},\frac{1}{p}\right)\in \mathscr{R}_1$. 

Let the region $\mathscr{R}_2$ be defined as follows,
\[\mathscr{R}_2:=\left\{(x,y,z):z\leq x+y<\min\left\{\frac{2d-1}{d},1+dz,z+\frac{2(d-1)}{d}\right\}\right\}.\]
We note that the boundedness $\mathfrak{M}_{loc}:L^{p_1}(\R^d)\times L^{p_2}(\R^d)\to L^p(\R^d)$ for the indices $p_1,p_2,p$ such that $\left(\frac{1}{p_1},\frac{1}{p_2},\frac{1}{p}\right)\in \mathscr{R}_2$ follows from Theorem \ref{JeongLee}. In fact, the boundedness of $\mathfrak{M}_{loc}$ in the region $\mathscr{R}_2$ can also be obtained by  choosing $r=1$ and $r=\infty$ and following the arguments of Proposition 3.2 in \cite{MaximalEstimatesForTheBilinearSphericalAveragesAndTheBilinearBochnerRieszOperators} verbatim. We skip the details.

To conclude, the boundedness of $\mathfrak{M}_{loc}$ in the region indicated in the hypothesis of Theorem \ref{slicedbilinearimproving} follows by interpolating the bounds for the regions $\mathscr{R}_1$ and $\mathscr{R}_2$.
\qed

\section{Linear intermediary spherical functions: Proof of Theorem \ref{linearAr}}\label{Sec:localmaximal}
\subsection{Proof of Theorem \ref{linearAr}:}
We employ a multiscale decomposition of the operator $\mathfrak{A}^{r}_*$. Let $\phi\in \mathcal S(\R^d)$ be a function such that $\widehat\phi$ is supported in $B(0,2)$ and $\widehat\phi(\xi)=1$ for $\xi\in B(0,1)$. We define
$\widehat{\phi_t}(\xi)=\widehat\phi(t\xi)$ and $\widehat{\psi}_t(\xi)= \widehat{\phi}(t\xi)-\widehat{\phi}(2t\xi)$. Then, we have the identity	
\begin{eqnarray}\label{identity}
\widehat \phi(\xi)+\sum_{j=1}^\infty\widehat \psi_{2^{-j}}(\xi)=1,\;\xi\neq 0.	
\end{eqnarray}
Using this identity, we have the following pointwise inequalities,
\[\mathfrak{A}^{r}f(x)\leq A^{r,0}_{1}f(x)+\sum_{j\geq1}A^{r,j}_{1}f(x),\;\;\text{and}\;\;\mathfrak{A}^{r}_*f(x)\leq A^{r,0}_{*}f(x)+\sum_{j\geq1}A^{r,j}_{*}f(x),\]
where
\[A^{r,j}_{*}f(x)=\sup_{k\in\Z}A_{2^k}^{r,j}f(x)=\sup_{k\in\Z}\|A_{2^kt}(f\ast\psi_{2^{k-j}})(x)\|_{L^r([1,2],t^{d-1}dt)},
\]
\[A^{r,0}_{*}f(x)=\sup_{k\in\Z}A_{2^k}^{r,0}f(x)=\sup_{k\in\Z}\|A_{2^kt}(f\ast\phi_{2^{k}})(x)\|_{L^r([1,2],t^{d-1}dt)}.\]
The operator $A_1^{r,j}$ has been studied extensively in \cite{SeegerVariation} to obtain variation estimates for spherical averages. We will require some $L^p$ estimates that are obtained by interpolating the  bounds for the endpoint $r=1,\infty$. Some of our bounds are already proved in \cite{SeegerVariation}, however we provide a proof for completeness.
\begin{lemma}
	Let $d\geq 2$ and $j\in\N$. For $1\leq r\leq\infty$, we have the following estimates,
	\begin{align}
	\|A^{r,j}_{1}\|_{L^1(\R^d)\to L^\infty(\R^d)}&\lesssim 2^\frac{j}{r'},\label{L1-infty-r}\\
	\|A^{r,j}_{1}\|_{L^1(\R^d)\to L^1(\R^d)}&\lesssim 2^\frac{j}{r'},\label{L1-1-r}\\
	\|A^{r,j}_{1}\|_{L^r(\R^d)\to L^r(\R^d)}&\lesssim 2^{-j\left(\frac{d-1}{r'}\right)},\;\;\;\;\;\;1\leq r\leq2\label{Lr-r-r},\\
	\|A^{r,j}_{1}\|_{L^r(\R^d)\to L^{r'}(\R^d)}&\lesssim 2^{-j\left(\frac{d-1}{r'}\right)},\;\;\;\;\;\;1\leq r\leq2\label{Lr-r'-r},\\
	\|A^{r,j}_{1}\|_{L^2(\R^d)\to L^2(\R^d)}&\lesssim 2^{-j\left(\frac{d-2}{2}+\frac{1}{r}\right)},\;\;2\leq r\leq\infty\label{L2-2-r}.
	\end{align}
\end{lemma}
\begin{proof}
	It is easy to see that
	\begin{align}
	\|A^{1,j}_{1}\|_{L^1(\R^d)\to L^1(\R^d)}&\lesssim 1,\label{L1-1-1}\\
	\|A^{1,j}_{1}\|_{L^1(\R^d)\to L^\infty(\R^d)}&\lesssim 1\label{L1-infty-1}.
	\end{align}
	By the kernel estimate $|\psi_{2^{-j}}*d\sigma_t(x)|\lesssim\frac{2^j}{(1+2^j||x|-t|)^N},$ for large $N$, we get
	\begin{align}
	\|A^{\infty,j}_{1}\|_{L^1(\R^d)\to L^1(\R^d)}&\lesssim 2^j,\label{L1-1-infty}\\
	\|A^{\infty,j}_{1}\|_{L^1(\R^d)\to L^\infty(\R^d)}&\lesssim 2^j\label{L1-infty-infty}.
	\end{align}
	The estimates \eqref{L1-infty-r} and \eqref{L1-1-r} follows by interpolating the bounds \eqref{L1-infty-1} with \eqref{L1-infty-infty} and \eqref{L1-1-1} with \eqref{L1-1-infty} respectively.
	
	Using the Fourier decay of the spherical measure $|\widehat{d\sigma}(\xi)|\lesssim (1+|\xi|)^{-\frac{d-1}{2}}$, we obtain
	\begin{equation}
	\|A^{2,j}_{1}\|_{L^2(\R^d)\to L^2(\R^d)}\lesssim 2^{-j\left(\frac{d-1}{2}\right)}.\label{L2-2-2}
	\end{equation}
	For $1\leq r \leq 2$, the estimate \eqref{Lr-r-r} follows directly from estimates \eqref{L1-1-1} and \eqref{L2-2-2}. Similarly, the estimate \eqref{Lr-r'-r} follows from the estimates \eqref{L1-infty-1} and \eqref{L2-2-2}.
	
	Moreover, by Stein's proof of spherical maximal function \cite{MaximalFunctionsISphericalMeans}, we have
	\begin{equation}
	\|A^{\infty,j}_{1}\|_{L^2(\R^d)\to L^2(\R^d)}\lesssim 2^{-j\left(\frac{d-2}{2}\right)}.\label{L2-2-infty}
	\end{equation}
	For $2\leq r \leq \infty$, the estimate \eqref{L2-2-r} follows by interpolating \eqref{L2-2-2} and \eqref{L2-2-infty}.
\end{proof}
We now state certain $L^p-$improving estimates for the single scale versions of the local spherical maximal operator $A_i^{\infty,j}$. For $d=2$, the estimates were obtained by Lee \cite{Lee1} by relying on local smoothing estimates and the case $d\geq3$ follows from the well-known Strichartz estimates. We refer to \cite{Lee1} for details.
\begin{lemma}[\cite{Lee1}]Let $1\leq r\leq \infty$. We have the following,
	\begin{enumerate}
		\item For $d\geq 3$, the following is true,
		\begin{align}
		\|A^{r,j}_{1}\|_{L^\frac{2r}{r+1}(\R^d)\to L^{\frac{2r'(d+1)}{d-1}}(\R^d)}&\lesssim 2^{-j\left(\frac{d^2-2d-1}{2r'(d+1)}\right)}\label{Lp03-q03-r}.
		\end{align}
		\item For $\frac{1}{p}+\frac{3}{q}=1$ and $q>\frac{14}{3}$, we have
		\begin{align}
		\|A^{r,j}_{1}\|_{L^\frac{pr}{p+r-1}(\R^2)\to L^{r'q}(\R^2)}&\lesssim 2^{\frac{j}{r'}\left(1-\frac{5}{q}\right)}\label{Lp02-q02-r}.
		\end{align}
	\end{enumerate}
\end{lemma}
\begin{proof}
	The estimate \eqref{Lp03-q03-r} follows by interpolating \eqref{L1-infty-1} and the Strichartz estimate \cite{Lee1} below,
	\[\|A^{\infty,j}_{1}\|_{L^2(\R^d)\to L^{\frac{2(d+1)}{d-1}}(\R^d)}\lesssim 2^{-j\left(\frac{d^2-2d-1}{2(d+1)}\right)}\label{L2-p0-infty}.\]
	The estimate \eqref{Lp02-q02-r} follows by using the bound \eqref{L1-infty-1} and the bound below, which was obtained in \cite{Lee1} by local smoothing estimates,
	\[\|A^{\infty,j}_{1}\|_{L^p(\R^2)\to L^{q}(\R^2)}\lesssim 2^{j\left(1-\frac{5}{q}\right)}\label{Lp-q-infty}.\]
\end{proof}
We now provide $L^p-$ bounds for the maximal operators $A^{r,j}_{*}$. The first is a direct consequence of estimates for the case $r=1,\infty$.
\begin{lemma}\label{L1Mj}
	Let $d\geq2$ and $j\geq 0$. Then for $f\in L^1_{loc}(\R^d)$, we have
	\[A^{r,j}_{*}f(x)\lesssim 2^{\frac{j}{r'}}M_{HL}f(x),\quad a.e.\;x\in\R^d.\]
\end{lemma}
Now, using the estimates for single scale operators $A_1^{r,j}$, we will prove some $L^p-$estimates for the maximal operators $A^{r,j}_{*}$ with norm depending on $j$. To do that, we rely on a interpolation scheme based on a vector-valued argument. We state Lemma \ref{vector} and the proof can be obtained by arguments similar to Lemma 5.4 of \cite{BCSSNikodym}.
\begin{lemma}\label{vector}
	Let $1\leq p_1,p_2\leq2$ be such that 
	\begin{align*}
	\|A^{r,j}_{1}\|_{L^{p_1}(\R^d)\to L^{p_1,\infty}(\R^d)}&\leq C_1,\\
	\|A^{r,j}_{*}\|_{L^{p_2}(\R^d)\to L^{p_2,\infty}(\R^d)}&\leq C_2.
	\end{align*}
	Then, we have
	\[\|A^{r,j}_{*}\|_{L^{p}(\R^d)\to L^{p}(\R^d)}\lesssim C_1^\frac{p_1}{2} C_2^{1-\frac{p_1}{2}},\;\;\text{for}\;p=\frac{2p_2}{2+p_2-p_1}.\]
\end{lemma}
\begin{lemma}
	Let $d\geq2$. The following holds true.
	\begin{itemize}
		\item  Let $1\leq r\leq2$. Then
		\begin{equation}\label{LpMj}
		\|A^{r,j}_{*}f(x)\|_{L^\frac{2}{3-r}(\R^d)}\lesssim 2^{-j\left(\frac{d(r-1)}{2}-\frac{1}{r'}\right)}\|f\|_{L^{\frac{2}{3-r}}(\R^d)}.
		\end{equation}
		\item Let $2\leq r\leq\infty$. Then
		\begin{equation}\label{L2Mj}
		\|A^{r,j}_{*}f(x)\|_{L^2(\R^d)}\lesssim 2^{-j\left(\frac{d-2}{2}+\frac{1}{r}\right)}\|f\|_{L^{2}(\R^d)}.
		\end{equation}
	\end{itemize}
\end{lemma}
\begin{proof}
	When $r\geq2$, the bound \eqref{L2Mj} follows from Littlewood-Paley theory and the estimate \eqref{L2-2-r}. For $1\leq r\leq2$, an application of Lemma \ref{vector} along with the estimates $\|A^{r,j}_{*}\|_{L^{1}\to L^{1,\infty}}\lesssim 2^\frac{j}{r'}$ (Lemma \ref{L1Mj}) and \eqref{Lr-r-r} implies the bound \eqref{LpMj}.
\end{proof}
We will require the interpolation trick of Bourgain that provides a restricted weak type estimate from two strong type bounds with appropriate growth and decay, see \cite{Lee1} for details.
\begin{lemma}{\cite{Lee1}}\label{interpolation}
	Let $\epsilon_1,\epsilon_2>0$. Suppose that $\{T_j\}$ is a sequence of $n-$linear (or sublinear) operators such that for some $1\leq p^i_1,p^i_2<\infty$, $i=1,2,\dots,n$ and $1\leq q_1,q_2<\infty$, 
	\begin{align}
	\Vert T_{j}(f^1,f^2,\dots,f^n)\Vert_{L^{q_1}(\R^d)}&\leq M_12^{\epsilon_1 j}\prod^{n}_{i=1}\Vert f^i\Vert_{L^{p^{i}_1}(\R^d)},\label{LeeLemma1}\\
	\Vert T_{j}(f^1,f^2,\dots,f^n)\Vert_{L^{q_2}(\R^d)}&\leq M_22^{-\epsilon_2 j}\prod^{n}_{i=1}\Vert f^i\Vert_{L^{p^{i}_2}(\R^d)}.\label{LeeLemma2}
	\end{align}
	Then $T=\sum_jT_j$ is bounded from $L^{p^1,1}(\R^d)\times L^{p^2,1}(\R^d)\times\cdots\times L^{p^{n},1}(\R^d)$ to $L^{q,\infty}(\R^d)$, i.e. 
	\begin{equation}
	\Vert T(f^1,f^2,\cdots,f^n)\Vert_{L^{q,\infty}(\R^d)}\lesssim M^{\theta}_{1}M^{1-\theta}_{2}\prod^{n}_{i=1}\Vert f^i\Vert_{L^{p^i,1}(\R^d)},\label{LeeLemma3}
	\end{equation}	
	where $\theta=\epsilon_2/(\epsilon_1+\epsilon_2)$, $1/q=\theta/q_1+(1-\theta)/q_2$, $1/r=\theta/r_1+(1-\theta)/r_2$ and $1/p^i=\theta/p^{i}_1+(1-\theta)/p^{i}_{2}$.
\end{lemma}
\begin{remark}
	We note that in the proof of Theorem \ref{linearAr}, we use the above lemma for the case when $q_1=\infty$. This can be justified as follows. We obtain an intermediate strong type estimate with growth in $j$ using real interpolation with the estimates \eqref{LeeLemma1} and \eqref{LeeLemma2}, and apply Lemma \ref{interpolation} for the intermediate estimate and the bound \eqref{LeeLemma2}. This process results in the same restricted weak type estimate as \eqref{LeeLemma3}.
\end{remark}
We now complete the proof of Theorem \ref{linearAr}.

\textbf{Proof of Theorem \ref{linearAr} \ref{stronglinearA}:} \sloppy It is clear that $\|\mathfrak{A}^r\|_{L^\infty(\R^d)\to L^\infty(\R^d)}\lesssim1$ and $\|\mathfrak{A}^r\|_{L^r(\R^d)\to L^\infty(\R^d)}\lesssim1$. Hence, by real interpolation, it is enough to prove the restricted weak type estimate for $\mathfrak{A}^r$ at the points $P=\left(\frac{r(d^2+1)}{d+1+r(d^2-d)},\frac{r'(d^2+1)}{d-1}\right),\;Q=\left(\frac{rd-r+1}{rd},\frac{1}{r'd}\right),$ and $R=\left(\frac{rd-r+1}{rd},\frac{rd-r+1}{rd}\right)$ (see Figure \ref{Fig:linearAr}). To achieve that, we will use the Lemma \ref{interpolation} along with the estimates,
\[\|A_1^{r,j}\|_{L^{p_1}(\R^d)\to L^{q_1}}(\R^d)\lesssim 2^{\epsilon_1 j},\;\;\; \|A_1^{r,j}\|_{L^{p_2}(\R^d)\to L^{q_2}(\R^d)}\lesssim 2^{-\epsilon_2 j},\]
with the necessary variables prescribed as in Table \ref{interpolationexponents}.
\begin{figure}[H]
	\begin{align*}
	\renewcommand{\arraystretch}{2}
	\begin{array}{|c|c|c|c|c|c|c|c|c|}
	\hline
	d & \text{Range of }r & \left(\frac{1}{p},\frac{1}{q}\right) &{p_1}& {q_1} & \epsilon_1 & {p_2}& {q_2} & \epsilon_2\\
	\hline
	\hline
	d=2 & [1,\infty] & P & 1 & \infty & \frac{1}{r'} \;\eqref{L1-infty-r}& \frac{19r}{12+7r} & \frac{19r'}{4} & \frac{1}{19r'}\;\;\;\;\;\;\;\eqref{Lp02-q02-r} \\
	\hline
	d\geq 3 & [1,\infty] & P & 1 & \infty & \frac{1}{r'}\;\eqref{L1-infty-r} & \frac{2r}{r+1} & \frac{2r'(d+1)}{d-1} & \frac{d^2-2d-1}{2r'(d+1)}\;\;\eqref{Lp03-q03-r} \\
	\hline
	d\geq 2 & [1,2] & Q & 1 & \infty & \frac{1}{r'}\;\eqref{L1-infty-r} & r & r' & \frac{d-1}{r'}\;\;\;\;\;\;\;\eqref{Lr-r'-r}\\
	\hline
	d\geq 2 & [2,\infty] & Q & 1 & \infty & \frac{1}{r'}\;\eqref{L1-infty-r} & 2 & 2 & \frac{d-2}{2}+\frac{1}{r}\;\eqref{L2-2-r}\\
	\hline
	d\geq 2 & [1,2] & R & 1 & 1 & \frac{1}{r'}\;\eqref{L1-1-r} & r & r & \frac{d-1}{r'}\;\;\;\;\;\;\;\eqref{Lr-r-r}\\
	\hline
	d\geq 2 & [2,\infty] & R & 1 & 1 & \frac{1}{r'}\;\eqref{L1-1-r} & 2 & 2 & \frac{d-2}{2}+\frac{1}{r}\;\eqref{L2-2-r} \\
	\hline
	\end{array}
	\end{align*}
	\caption{The table prescribes the values of the exponents used for the interpolation Lemma \ref{interpolation} to obtain restricted weak type bounds for $\mathfrak{A}^r$.}
	\label{interpolationexponents}
\end{figure}

\textbf{Proof of Theorem \ref{linearAr}\ref{restrictedlinearA}:} It is enough to prove the restricted weak type inequality for $\mathfrak{A}_*^r$ at the point $\left(\frac{rd}{rd-r+1},\frac{rd}{rd-r+1}\right)$; For $p>\frac{rd}{rd-r+1}$, the $L^p$ boundedness for $\mathfrak A_*^r$ follows by interpolation.

For $r\geq 2$, we rely on estimates \eqref{L2Mj} and $\mathfrak A_*^r:L^1(\R^d)\to L^{1,\infty}(\R^d)$ (Lemma \ref{L1Mj}). However, the interpolation Lemma \ref{interpolation} is not applicable in this case. We can get around this by interpolating the weak $(1,1)$ estimate with the strong $(2,2)$ estimate to control the quantity $\|\mathfrak A_*^r\|_{L^{p_0}(\R^d)\to L^{p_0}(\R^d)}$ for some $1<p_0<\frac{rd}{rd-r+1}$. Now we can obtain the desired restricted weak type by using Lemma \ref{interpolation} for the strong $(p_0,p_0)$ and $(2,2)$ bounds. The case $1\leq r\leq 2$ can be resolved similarly by using the estimate \eqref{LpMj} instead of \eqref{L2Mj}. We leave the details to the reader.\\

\textbf{Proof of Theorem \ref{linearAr} \ref{necessarylinearAr}:} Let $\delta>0$ be a small number and $c>1$ be a fixed constant. We define $B(0,\delta)$ to be ball with center $0\in\R^d$ and radius $\delta$, and $S^{\delta}(0)$ to be the $\delta-$neighborhood of $S^{d-1}$, i.e. $\{x\in\R^d:||x|-1|<\delta\}$. We also define $R_1$ to be the rectangle of dimension $[-c\sqrt{\delta},c\sqrt{\delta}]^{d-1}\times[-c\delta,c\delta]$, centered at origin with smaller side along the direction of $e_d=(0,0,\dots,0,1)$ and $R_2$ to be the rectangle of dimension $[-\sqrt{\delta},\sqrt{\delta}]^{d-1}\times[1,2]$ centered at $(0,0,\dots,0,\frac{3}{2})$, with longer side along the direction of $e_d$.

We define function $f$ with $\|f\|_{L^p}\simeq\delta^\alpha$ and a test set $E$ with $|E|\simeq\delta^\beta$ satisfying 
\[\|A_tf\|_{L^r([1,2])}\gtrsim \delta^\gamma, \;\;\;\text{for all}\;x\in E.\]
Since $\delta>0$ is a small number, the required necessary condition holds if
\[\alpha \leq \frac{\beta}{q}+\gamma.\]
\begin{figure}[H]
	\begin{align*}
	\renewcommand{\arraystretch}{2}
	\begin{array}{|l|c|c|c|c|c|}
	\hline
	\text{$f$ (function)} & E \text{(test set)} & \alpha & \beta & \gamma & \text{Necessary Condition}\\
	\hline
	\hline
	\chi_{S^{c\delta}(0)} & B(0,\delta) & \delta^{\frac{1}{p}} & \delta^{d} & \delta^{\frac{1}{r}} & \frac{1}{p}\leq\frac{d}{q}+\frac{1}{r}\\
	\hline
	\chi_{B(0,c\delta)}& B(0,2)\setminus B(0,1) & \delta^{\frac{d}{p}} & 1 & \delta^{d-1+\frac{1}{r}} & \frac{d}{p}\leq d-1+\frac{1}{r}\\
	\hline
	\chi_{R_1} & R_2 & \delta^{\frac{d+1}{2p}} & \delta^{\frac{d-1}{2}} & \delta^{\frac{d-1}{2}+\frac{1}{r}} & \frac{d+1}{2p}\leq\frac{d-1}{2q}+\frac{1}{r}+\frac{d-1}{2}\\
	\hline
	\chi_{B(0,\frac{1}{\delta)}}& B(0,\frac{1}{\delta}) & \delta^{\frac{-d}{p}} & \delta^{-d} & 1 & \frac{1}{q}\leq \frac{1}{p}\\
	\hline
	\end{array}
	\end{align*}
	\caption{The table provides examples for the necessary conditions required for boundedness of the operator $\mathfrak A^r$.}\label{examples-A}
\end{figure}
The functions and test sets are given in the Figure \ref{examples-A}.\qed

As mentioned in Remark \ref{Rem:strongtype}, we provide an example showing that $\mathfrak{A}^r$ does not map $L^p(\R^d)$ to $L^q(\R^d)$ when $\left(\frac{1}{p},\frac{1}{q}\right)$ lies on the line segment $QR$.
\begin{proposition}\label{restrictedexample}
	The operator $\mathfrak{A}^{r}$ does not map $L^{\frac{dr}{dr-r+1},s}(\R^d)$ to  $L^{q}(\R^d)$, for any $s>1$ and $q\geq1$.
\end{proposition}
\begin{proof}
	Let $p_0=\frac{dr}{dr-r+1}$. For a small number $a>0$, we define 
	\[f(x)=\sum_{i=1}^{N}4^{\frac{d}{p_0}i}\chi_{B(0,a4^{-i})}(x).\]
	For $s>1$, we have the following bound on Lorentz space norm of $f$,
	\[\|f\|_{L^{p_0,s}}\lesssim N^{\frac{1}{s}}.\]
	To see this, consider the sets $F_j=\bigg(4^{\frac{d}{p_0}j}\sum\limits_{k=0}^{j-1}4^{-\frac{d}{p_0}k},4^{\frac{d}{p_0}(j+1)}\sum\limits_{k=0}^{j}4^{-\frac{d}{p_0}k}\bigg]$ for $1\leq j\leq N-1$. If $t\in F_j$ , we have 
	\[\{x\in\R^d:\;|f(x)|>t\}=B(0,a4^{-(j+1)}).\]
	Denote $d_f(t)=|\{x\in\R^d:\;|f(x)|>t\}|$ and observe that, $$td_f(t)^\frac{1}{p_0}=t(a4^{-(j+1)d})^\frac{1}{p_0}\lesssim1.$$
	Therefore, we have 
	\begin{align*}
	\|f\|_{L^{p_0,s}}&= p_0^\frac{1}{s}\left(\int_0^{4^{\frac{d}{p_0}}}[d_f(t)t]^s\frac{dt}{t}+\sum_{j=0}^{N-1}\int_{F_j}[d_f(t)t]^s\frac{dt}{t}\right)^\frac{1}{s}\\
	&\lesssim \left(\int_0^{4^{\frac{d}{p_0}}}4^{-\frac{d}{p_0}}t^{s-1}\;dt+\sum_{j=0}^{N-1}\int_{F_j}\frac{dt}{t}\right)^\frac{1}{s}\\
	&\leq \left(4^{\frac{d}{p_0}(s-1)}+\sum_{j=0}^{N-1}1\right)^\frac{1}{s}\lesssim N^\frac{1}{s}.
	\end{align*}
	
	Now for $x\in \{y:\;1\leq|y|\leq2\}$, we have
	\[\left(\int_1^2|A_tf(x)|^r\right)^\frac{1}{r}\gtrsim\sum_{i=1}^{N}4^{\frac{d}{p_0}i}4^{-\left(d-1+\frac{1}{r}\right)}=N.\]
	If $\mathfrak{A}^{r}$ maps $L^{\frac{dr}{dr-r+1},s}(\R^d)$ to  $L^{q}(\R^d)$, then
	\begin{eqnarray*}
		N\lesssim\|\mathfrak{A}^{r}f\|_{L^{q}}\lesssim\|\mathfrak{A}^{r}f\|_{L^{p_0,s}\to L^{q}}\|f\|_{p_0,s}\lesssim \|\mathfrak{A}^{r}f\|_{L^{p_0,s}\to L^{q}}N^{\frac{1}{s}}.
	\end{eqnarray*}
	which is a contradiction for $s>1$.
\end{proof}

\section{Proof of Theorems \ref{Mfull} and \ref{linearized}}\label{Mfullsec}
\subsection{Proof of Theorem \ref{Mfull}:}
	We will in fact provide a counterexample for the local maximal operator defined by
	\[{\mathcal  M}_{loc}^\theta (f,g)(x)=\sup_{t\in[1,2]} |\mathcal A_t^\theta(f,g)(x)|,\]
	The proof is based on the Kakeya construction used in \cite{EndpointMappingPropertiesOfSphericalMaximalOperators} to show that the spherical maximal function does not map $L^{2,1}(\R^2)\to L^{2,\infty}(\R^2)$. Let $R_l$ be the collection of $\delta^{-1}$ overlapping rectangles of dimension $\delta\times \delta^2$ lying inside the cube $[-\delta,\delta]^2$ such that the longer side of $R_l$ is parallel to the vector $e^{i\delta l}$. We have $|\cup_l R_l|\sim \frac{\delta^{2}}{\log\frac{1}{\delta}}$. Also let $[1,2]=\cup_\nu I_\nu$, where $I_\nu$'s are $\delta^{-2}$ disjoint intervals of equal length. We denote $R_{l,\nu}$ to be the rectangle obtained by translating the rectangle $R_l$ by length $I_\nu$ along its shorter side. Also let $\widetilde R_{l,\nu}$ be the rectangle obtained by translating the rectangle $R_l$ by length $2I_\nu$ along its shorter side, followed by a counterclockwise rotation by the angle $\theta$.  We define
	\[f(x)=\chi_{\bigcup\limits_l R_l}(x),\;\;\;\;g(x)=\chi_{\bigcup\limits_{l,\nu}\widetilde R_{l,\nu}}(x).\]
	We have $\|f\|_{p_1}\sim(\frac{\delta^{2}}{\log\frac{1}{\delta}})^{\frac{1}{p_1}}$ and $\|g\|\sim1$. For $x\in\bigcup\limits_{l,\nu}R_{l,\nu}$, it follows that $\mathcal{M}_{loc}^\theta(f,g)(x)> c\delta$ for some absolute constant $c>0$. Thus,
	\[|\{{\mathcal{M}}_{loc}^\theta(f,g)(x)>c\delta\}|\gtrsim 1\gtrsim \frac{\delta^{-p(\frac{2}{p_1}-1)}\log\frac{1}{\delta}^{\frac{p}{p_1}}}{\delta^{p}}\|f\|_{p_1}^p\|g\|_{p_2}^p.\]
	Therefore, we get that 
	$$\|{\mathcal{M}}_{loc}^\theta\|_{L^{p_1,1}\times L^{p_2,1}\to L^{p,\infty}}=\infty~\text{for}~ p_1\leq2.$$ 
	By symmetry, the same holds for $p_2\leq 2$.
\qed

\subsection{Proof of Theorem \ref{linearized}:}
	We will prove the theorem for the case $\theta=\pi$, the proof of other cases is similar.
	\begin{align*}
		&\langle \widetilde{\mathcal A}^\pi(f,g),h\rangle\\
		=&\int\int f(x+|x|y)g(x-|x|y)h(x)\;d\sigma(y)dx\\
		=&\int_{r=0}^\infty\int_{\theta=0}^{2\pi}\int_{t=0}^{2\pi} f(r(e^{i\theta}+e^{it}))g(r(e^{i\theta}-e^{it}))h(re^{i\theta})\;dtrdrd\theta\\
		=&\int_{r=0}^\infty\int_{\theta=0}^{2\pi}\int_{t=\theta}^{2\pi+\theta} f\left(2r\cos\left(\frac{\theta-t}{2}\right)e^{i\frac{\theta+t}{2}}\right)g\left(2r\sin\left(\frac{\theta-t}{2}\right)e^{i(\frac{\theta+t}{2}+\frac{\pi}{2})}\right)h(re^{i\theta})\;dtrdrd\theta\\
	\end{align*}
	By the change of variable $u=\cos\left(\frac{\theta-t}{2}\right)$, the above term is equal to
	\begin{align*}
		=&\int_{r=0}^\infty\int_{\theta=0}^{2\pi}\int_{u=-1}^1f(2rue^{i(\theta-\cos^{-1} u)})g(2r\sqrt{1-u^2}e^{i(\theta-\cos^{-1} u+\frac{\pi}{2})})h(re^{i\theta})\;2\frac{du}{\sqrt{1-u^2}}rdrd\theta\\
		=&\int_{u=-1}^1\int_{\R^2}f(2ux)g(2\sqrt{1-u^2}\mathfrak{R}_{-\frac{\pi}{2}}x)h(\mathfrak{R}_{\cos^{-1} u}x)\;dx\frac{2du}{\sqrt{1-u^2}},
	\end{align*}
	where $\mathfrak{R}_\phi x$ is the point obtained by rotating $x$ by angle $\phi$ counterclockwise. Applying H\"older's inequality and scaling, we get the above quantity is dominated by
	\begin{align*}
		\lesssim& \int_{u=-1}^1\|f(2u\cdot)\|_{p_1}\|g(2\sqrt{1-u^2}\cdot)\|_{p_2}\|h\|_{p'}\frac{du}{\sqrt{1-u^2}}\\
		\lesssim&\|f\|_{p_1}\|g\|_{p_2}\|h\|_{p'}\int_{u=0}^1u^{-\frac{2}{p_1}}(1-u^2)^{-\frac{1}{p_2}-\frac{1}{2}}\;du\\
		\lesssim&\|f\|_{p_1}\|g\|_{p_2}\|h\|_{p'}\int_{t=0}^1t^{-\frac{1}{p_1}-\frac{1}{2}}(1-t)^{-\frac{1}{p_2}-\frac{1}{2}}\;dt,
	\end{align*}
	where the beta integral in the above quantity is finite for $p_1,p_2>2$ and the proof concludes.

We now prove the restricted weak type estimates for $\widetilde{\mathcal A}^\pi$. We observe that $\widetilde{\mathcal A}^\pi:L^\infty(\R^2)\times L^{2,1}(\R^2)\to L^{2,\infty}(\R^2)$ and $\widetilde{\mathcal A}^\pi:L^{2,1}(\R^2)\times L^\infty(\R^2)\to L^{2,\infty}(\R^2)$ follows from the inequality,
\[\widetilde{\mathcal A}^\pi(f,g)(x)\leq\min\{\|f\|_\infty \widetilde A(|g|)(x),\|g\|_\infty \widetilde A(|f|)(x)\}.\]
We prove the restricted weak type inequality at the endpoint $(2,2,1)$, the proof of the remaining endpoints are similar. We decompose the operator  $\widetilde{\mathcal A}^\pi$ as follows 
	\begin{eqnarray*}
		\widetilde{\mathcal A}^\pi(f,g)(x)&=&\int_{|u|\leq1/2}f(2uxe^{-\iota\cos^{-1}u})g(2x\sqrt{1-u^2}e^{-\iota\cos^{-1}u+\iota\pi/2})~\frac{2du}{\sqrt{1-u^2}}\\
		&+&\int_{1/2<|u|\leq1}f(2uxe^{-\iota\cos^{-1}u})g(2x\sqrt{1-u^2}e^{-\iota\cos^{-1}u+\iota\pi/2})~\frac{2du}{\sqrt{1-u^2}}\\
		&:=&I_0(f,g)(x)+I_{1}(f,g)(x).
	\end{eqnarray*}
	We further decompose each of the two operators into infinitely many pieces as follows 
	\begin{eqnarray*}
		I_0(f,g)(x)&=&\sum_{j\geq1}\int_{2^{-j-1}<|u|\leq2^{-j}}f(2uxe^{-\iota\cos^{-1}u})g(2x\sqrt{1-u^2}e^{-\iota\cos^{-1}u+\iota\pi/2})~\frac{2du}{\sqrt{1-u^2}}\\
		&:=&\sum_{j\geq1}I_{0,j}(f,g)(x).	
	\end{eqnarray*}
	Note that the denominator $\sqrt{1-u^2}$ in the expression above behaves like a constant as $|u|\leq1/2$. We have 
	\begin{eqnarray}
		\Vert I_{0,j}(f,g)\Vert_{L^{1}}&\lesssim& \int_{2^{-j-1}<|u|\leq2^{-j}}\Vert f(2u\cdot)\Vert_{L^{4}}\Vert g(2\sqrt{1-u^2}\cdot)\Vert_{L^{4/3}} ~du \\
		&\lesssim&2^{-j/2}\Vert f\Vert_{L^{4}}\Vert g\Vert_{L^{4/3}}.
	\end{eqnarray}
	On the other hand,
	\begin{eqnarray}
		\Vert I_{0,j}(f,g)\Vert_{L^{1}}&\lesssim& \int_{2^{-j-1}<|u|\leq2^{-j}}\Vert f(2u\cdot)\Vert_{L^{4/3}}\Vert g(2\sqrt{1-u^2}\cdot)\Vert_{L^{4}} ~du \\
		&\lesssim&2^{j/2}\Vert f\Vert_{L^{4/3}}\Vert g\Vert_{L^{4}}.
	\end{eqnarray}
	Applying Bourgain's interpolation trick (Lemma \ref{interpolation}) we get that the operator $I_0$ maps $L^{2,1}\times L^{2,1}$ to $L^{1,\infty}$.
	
	Next, consider a similar decomposition of $I_1(f,g)$ as follows 
	\begin{eqnarray*}
		I_1(f,g)(x)&=&\sum_{j\geq1}\int_{1-2^{-j}<|u|\leq 1-2^{-j-1}}f(2uxe^{-\iota\cos^{-1}u})g(2x\sqrt{1-u^2}e^{-\iota\cos^{-1}u+\iota\pi/2})~\frac{2du}{\sqrt{1-u^2}}\\
		&:=&\sum_{j\geq1}I_{1,j}(f,g)(x).
	\end{eqnarray*}
	Now computing the $L^{1}-$norm we get 
	\begin{eqnarray}
		\Vert I_{1,j}(f,g)\Vert_{L^{1}}&\lesssim& \int_{1-2^{-j}<|u|\leq1-2^{-j-1}}\Vert f(2u\cdot)\Vert_{L^{4}}\Vert g(2\sqrt{1-u^2}\cdot)\Vert_{L^{4/3}} ~\frac{du}{\sqrt{1-u^2}} \\
		&\lesssim&2^{j/4}\Vert f\Vert_{L^{4}}\Vert g\Vert_{L^{4/3}}.
	\end{eqnarray}
	And \begin{eqnarray}
		\Vert I_{1,j}(f,g)\Vert_{L^{1}}&\lesssim& \int_{1-2^{-j}<|u|\leq1-2^{-j-1}}\Vert f(2u\cdot)\Vert_{L^{4/3}}\Vert g(2\sqrt{1-u^2}\cdot)\Vert_{L^{4}} ~\frac{du}{\sqrt{1-u^2}} \\
		&\lesssim&2^{-j/4}\Vert f\Vert_{L^{4/3}}\Vert g\Vert_{L^{4}}.
	\end{eqnarray}
	Applying the interpolation lemma we get that the operator $I_1$ maps $L^{2,1}\times L^{2,1}$ to $L^{1,\infty}$. Finally, combining the estimates of both $I_0$ and $I_1,$ we get the desired result.
\qed
\iffalse
{\color{red}\begin{proof}[Proof of Proposition \ref{A1estimate}]
	We first claim that $\mathcal{A}_{1}$ satisfies $L^{\frac{d+1}{d}}\times L^{\frac{d+1}{d}}\to L^{1}$ boundedness.
	\begin{eqnarray*}
		\|\mathcal A (f,g)\|_{1}&=&\int_{\R^d}|\int_{\s^{d-1}}f(x-y)g(x+y)\;d\sigma(y)|\;dx\\
		&\leq& \int_{\s^{d-1}}\int_{\R^d}|f(x-y)||g(x+y)|\;dx\;d\sigma(y)\\
		&=& \int_{\R^d}|f(x)|\int_{\s^{d-1}}|g(x+2y)|\;d\sigma(y)\;dx\\
		&\leq& \|f\|_{\frac{d+1}{d}}\|\int_{\s^{d-1}}|g(x+2y)|\;d\sigma(y)\|_{d+1}\leq \|f\|_\frac{d+1}{d}\|g\|_{\frac{d+1}{d}}.
	\end{eqnarray*}
	Secondly, we claim that  $\mathcal{A}_{1}$ satisfies
	$L^{\frac{d+1}{d}}\times L^{\infty}\to L^{d+1}$ and $L^{\infty}\times L^{\frac{d+1}{d}}\to L^{d+1}$ estimates.
	\begin{eqnarray*}
		\|\mathcal A (f,g)\|_{d+1}&=&(\int_{\R^d}|\int_{\s^{d-1}}f(x-y)g(x+y)\;d\sigma(y)|^{d+1}\;dx)^{\frac{1}{d+1}}\\
		&\leq& \|g\|_{\infty}\|\int_{\s^{d-1}}|f(x-y)|\;d\sigma(y)\|_{d+1}\leq \|f\|_\frac{d+1}{d}\|g\|_{\infty}.
	\end{eqnarray*}
	Finally, we claim that  $\mathcal{A}_{1}$ satisfies
	$L^{1}\times L^{\infty}\to L^{1}$ and $L^{\infty}\times L^{1}\to L^{1}$ estimates.
	\begin{eqnarray*}
		\|\mathcal A (f,g)\|_{1}&=&\int_{\R^d}|\int_{\s^{d-1}}f(x-y)g(x+y)\;d\sigma(y)|\;dx)\\
		&\leq& \|g\|_{\infty}\|\int_{\s^{d-1}}|f(x-y)|\;d\sigma(y)\|_{1}\leq \|f\|_{1}\|g\|_{\infty}.
	\end{eqnarray*}
	Now, applying real interpolation we get the desired boundedness.
\end{proof}}
\fi

\section{Necessary conditions for $\mathcal{M}$ in dimensions $d\geq2$}\label{Sec:Necessary}
In this section we provide some necessary conditions for the higher dimensional analogue $\mathcal{T}$ of the operator $\mathcal{A}_1^\pi$ defined as,
\[\mathcal{T}(f,g)(x)=\int_{\s^{d-1}}f(x-y)g(x+y)\;d\sigma(y).\]
The first result concerns an generalization of the necessary condition \eqref{necessarycondition} for $\mathcal{T}$. We note that the condition is obtained by considering functions of product type instead of examples generated by C. Fefferman boxes \cite{Fefferman}, as was the case in \cite{GIKL}.
	\begin{proposition}
		Let $d\geq 2$ and $1\leq p_1,p_2<\infty$. Suppose $\mathcal{T}$ satisfies the following inequality,
		\[\|\mathcal{T}(f,g)\|_{L^p(\R^d)}\lesssim\|f\|_{L^{p_1}(\R^d)}\|g\|_{L^{p_2}(\R^d)},\]
		for functions $f,g$ of the form $f(x)=f_1(x_1)f_2(x_2)$ and $g(x)=g_1(x_1)g_2(x_2)$ where we write $x=(x_1,x_2)$ with $x_1\in\R^{d_1}$, $x_2\in\R^{d_2}$ and $d=d_1+d_2$. Then we have,
		\[\frac{d+1}{p_1}+\frac{d+1}{p_2}\leq d-1+\frac{d+1}{p}.\]
	\end{proposition}
	\begin{proof}
	Consider the functions 
	\[f(x)={||x_1|-1|^{-\frac{d_1\alpha_1}{p_1}}}{|x_2|^{-\frac{d_2\alpha_2}{p_1}}}\chi_{[0,1]^d}(x)\;\text{ and }\; g(x)={||x_1|-1|^{-\frac{d_1\beta_1}{p_2}}}{|x_2|^{-\frac{d_2\beta_2}{p_2}}}\chi_{[0,1]^d}(x).\]
	Then $f\in L^{p_1}(\R^d)$ if $\alpha_1<\frac{1}{d_1}$ and $\alpha_2<1$. Similarly, $g\in L^{p_2}(\R^d)$ if $\beta_1<\frac{1}{d_1}$ and $\beta_2<1$. By the slicing argument and decomposing the interval $[0,1]$ into dyadic annulli, we get that
	\begin{eqnarray*}
		&&\mathcal T (f,g)(x)\\
		&=&\int_{B^{d_1}(0,1)}f(x_1-y_1)g(x_1+y_1)(1-|y_1|^2)^{\frac{d_2-2}{2}}\\
		&&\hspace{10mm}\int_{\s^{d_2-1}}f(x_2-\sqrt{1-|y_1|^2}y_2)g(x_2+\sqrt{1-|y_1|^2}y_2)~d\sigma(y_2)dy_1\\
		&=&\int_0^1(1-r^2)^{\frac{d_2-2}{2}}r^{d_1-1}\left(\int_{\s^{d_1-1}}f(x_1-ry_1)g(x_1+ry_1)~d\sigma(y_1)\right)\\
		&&\hspace{10mm}\left(\int_{\s^{d_2-1}}f(x_2-\sqrt{1-r^2}y_2)g(x_2+\sqrt{1-r^2}y_2)~d\sigma(y_2)\right)dr\\
		&=&\sum_{j\geq1}\int_{1-2^{-j+1}}^{1-2^{-j}}(1-r^2)^{\frac{d_2-2}{2}}r^{d_1-1}\left(\int_{\s^{d_1-1}}{||x_1-ry_1|-1|^{-\frac{d_1\alpha_1}{p_1}}}{||x_1+ry_1|-1|^{-\frac{d_1\beta_1}{p_2}}}~d\sigma(y_1)\right)\\
		&&\hspace{10mm}\left(\int_{\s^{d_2-1}}{|x_2-\sqrt{1-r^2}y_2|^{-\frac{d_2\alpha_2}{p_1}}}{|x_2+\sqrt{1-r^2}y_2|^{-\frac{d_2\beta_2}{p_2}}}~d\sigma(y_2)\right)dr.
	\end{eqnarray*}
	
	Let $B_k=\{x=(x_1,x_2):2^{-k}\leq|x_1|\leq 2^{-k+1},2^{-\frac{k}{2}}\leq|x_2|\leq 2^{-\frac{k-1}{2}}$, then for large $k$ and $j>k$, we get that
	\[||x_1\pm ry_1|-1|\sim 2^{-k} ~~\text{and}~~ |x_2\pm\sqrt{1-r^2}y_2|\sim 2^{-\frac{k}{2}}.\]
	We can see that
	\begin{eqnarray*}
		\mathcal T (f,g)(x)&\geq& \sum_{j\geq k}\int_{1-2^{-j+1}}^{1-2^{-j}}(1-r^2)^{\frac{d_2-2}{2}}r^{d_1-1}\left(\int_{\s^{d_1-1}}{2^{k\frac{d_1\alpha_1}{p_1}}}{2^{k\frac{d_1\beta_1}{p_2}}}~d\sigma(y_1)\right)\\
		&&\hspace{10mm}\left(\int_{\s^{d_2-1}}{2^{k\frac{d_2\alpha_2}{2p_1}}}{2^{k\frac{d_2\beta_2}{2p_2}}}~d\sigma(y_2)\right)dr\\
		&\gtrsim& 2^{k(\frac{d_1\alpha_1}{p_1}+\frac{d_1\beta_1}{p_2}+\frac{d_2\alpha_2}{2p_1}+\frac{d_2\beta_2}{2p_2})}\sum_{j\geq k}2^{-j\frac{d_2}{2}}=2^{k(\frac{d_1\alpha_1}{p_1}+\frac{d_1\beta_1}{p_2}+\frac{d_2\alpha_2}{2p_1}+\frac{d_2\beta_2}{2p_2}-\frac{d_2}{2})}.
	\end{eqnarray*}
	
	Hence, we get that
	\begin{eqnarray*}
		\|\mathcal T (f,g)\|_p^p&\geq& \sum_k\int_{B_k}|\mathcal T (f,g)(x)|^p~dx\\
		&\gtrsim& \sum_k\int_{B_k}2^{kp(\frac{d_1\alpha_1}{p_1}+\frac{d_1\beta_1}{p_2}+\frac{d_2\alpha_2}{2p_1}+\frac{d_2\beta_2}{2p_2}-\frac{d_2}{2})}~dx\\
		&=&\sum_k 2^{kp(\frac{d_1\alpha_1}{p_1}+\frac{d_1\beta_1}{p_2}+\frac{d_2\alpha_2}{2p_1}+\frac{d_2\beta_2}{2p_2}-\frac{d_2}{2})} 2^{-k(d_1+\frac{d_2}{2})}.
	\end{eqnarray*}
	The above sum is finite if $\frac{d_2+2}{2p_1}+\frac{d_2+2}{2p_2}\leq \frac{d_2}{2}+\frac{2d_1+d_2}{2p}$. This condition implies the  necessary condition if we choose $d_1=1$ and $d_2=d-1$.
	\end{proof}
	We also record some $L^p-$ improving conditions for $\mathcal{T}$ in the following proposition. These conditions are higher dimensional analogues of results obtained in \cite{GIKL} by considering indicator functions of appropriate balls and annulus. 
	\begin{proposition}
		Let $d\geq 2$ and $1\leq p_1,p_2<\infty$. Suppose $\mathcal{T}:L^{p_1}(\R^d)\times L^{p_2}(\R^d)\to L^p(\R^d)$ boundedly, Then we have the following,
		\begin{enumerate}
			\item $\frac{d}{p_1}+\frac{1}{p_2}\leq 1+\frac{1}{p}$.
			\item $\frac{1}{p_1}+\frac{d}{p_2}\leq 1+\frac{1}{p}$.
			\item $\frac{1}{p}\leq \frac{1}{p_1}+\frac{1}{p_2}\leq \frac{d}{p}$.
		\end{enumerate}
	\end{proposition}
	We refer to Section 3 of \cite{GIKL} for details.

\section*{Acknowledgements}

We thank the anonymous referee for their helpful suggestions and comments regarding the article. We would also like to thank Yumeng Ou, Tainara Borges and Benjamin Foster for pointing out an error in the previous version of this article. The authors thank Michael Lacey and Ben Krause for introducing the bilinear operator $\mathcal{A}_t^\theta$ to them. Ankit Bhojak and Saurabh Shrivastava acknowledge the financial support from Science and Engineering Research Board, Department of Science and Technology, Govt. of India, under the scheme Core Research Grant, file no. CRG/2021/000230. Surjeet Singh Choudhary is supported by CSIR(NET), file no.09/1020(0182)/2019- EMR-I for his Ph.D. fellowship. Kalachand Shuin is supported by NRF grant no. 2022R1A4A1018904 and BK 21 Post doctoral fellowship of South Korea. The authors acknowledge the support and hospitality provided by the International Centre for Theoretical Sciences, Bangalore (ICTS) for participating in the program - Modern trends in Harmonic Analysis (code: ICTS/Mtha2023/06).  
\bibliography{bibliography}
\end{document}